\documentclass[12pt]{amsart}
\usepackage[a4paper,scale=.841,centering]{geometry}  
\usepackage{mathrsfs,amssymb}
\usepackage[mathlines]{lineno}
\usepackage[dvipsnames]{xcolor}
\usepackage[
pagebackref,
colorlinks,citecolor=Plum,urlcolor=Periwinkle,linkcolor=DarkOrchid]{hyperref}
\usepackage{graphicx}
\usepackage{subfig}
 \theoremstyle{plain}
\newtheorem{theorem}{Theorem}
\newtheorem{model}{Model}
\newtheorem{lemma}[theorem]{Lemma}
\newtheorem{corollary}[theorem]{Corollary}
\newtheorem{proposition}[theorem]{Proposition}

\theoremstyle{definition}
\newtheorem*{definition}{Definition}
\newtheorem*{remark}{Remark}

\newcommand{\cadlag}{c\`adl\`ag}
\newcommand{\defin}[1]{{\bf #1}}

\newcommand{\re}{\ensuremath{\mathbb{R}}}
\newcommand{\paren}[1]{\ensuremath{\left( #1\right) }}
\newcommand{\F}{\ensuremath{\mathscr{F}}}

\newcommand{\p}{\mathbb{P}}
\newenvironment{esn}{\begin{equation*}}{\end{equation*}}
\newcommand{\imf}[2]{\ensuremath{#1\!\paren{#2}}}
\newcommand{\se}{\ensuremath{\mathbb{E}}}
\newcommand{\cond}[2]{\left.\vphantom{#2}#1\ \right| #2}

\newcommand{\bra}[1]{\ensuremath{\left[ #1\right] }}

\newcommand{\na}{\ensuremath{\mathbb{N}}}
\newcommand{\set}[1]{\ensuremath{\left\{ #1\right\} }}
\newcommand{\indi}[1]{\si_{#1}}
\newcommand{\si}{{\ensuremath{\bf{1}}}}

\newcommand{\realtree}{\ensuremath{\re}\nbd tree}
\newcommand{\bb}[1]{\mathbb{#1}}
\newcommand{\sa}{\ensuremath{\sigma}\nbd field}
\newcommand{\nbd}{\nobreakdash -}
\newcommand{\leb}{\text{Leb}}
\newcommand{\clo}[1]{\ensuremath{\overline{#1}}}
\newcommand{\abs}[1]{\hspace{.25mm}\left|#1\right|\hspace{.25mm}}
\newcommand{\imi}[2]{#2^{-1}\!\paren{#1}}
\newcommand{\ball}[2]{\imf{B_{#1}}{#2}}
\newcommand{\q}{\ensuremath{ \bb{Q}  } }
\newcommand{\sag}[1]{\sigma\!\paren{#1}}
\newcommand{\ra}{\ensuremath{\mathbb{Q}}}

\newcommand{\mc}[1]{\ensuremath{\mathscr{#1}}}
\newcommand{\G}{\ensuremath{\mc{G}}}
\newcommand{\floor}[1]{\ensuremath{\lfloor #1\rfloor}}
\newcommand{\cts}{\ensuremath{\mathbf{C}}}

\newcommand{\tr}[1]{\ensuremath{\mathbf{#1}}}

\newcommand{\cemetery}{\dagger}

\newcommand{\eps}{\ensuremath{ \varepsilon}}
\newcommand{\fun}[3]{\ensuremath{#1:#2\to #3}}
\newcommand{\fund}[3]{\ensuremath{#1:#2\mapsto #3}}

\title{Totally ordered measured trees and splitting trees with infinite variation}
\author{Amaury Lambert}
\address{Laboratoire de Probabilit\'es et Mod\`eles Al\'eatoires\\
UPMC Univ Paris  06\\
Case courrier 188\\
4, Place Jussieu\\
75252 PARIS Cedex 05}
\author{Ger\'onimo Uribe Bravo}
\address{Instituto de Matem\'aticas\\ UNAM \\
\'Area de la Investigaci\'on Cient\'ifica, Circuito Exterior, Ciudad Universitaria\\ Coyoac\'an, 04510. M\'exico, D. F. }
\subjclass[2010]{
60G51
, 60J80
, 05C05 
, 92D25 
}
\thanks{GUB's research is supported by UNAM-DGAPA-PAPIIT grant no. IA101014. 
AL thanks the \emph{Center for Interdisciplinary Research in Biology} (Coll\`ege de France, Paris) for funding.} 
\usepackage[dvipsnames]{xcolor}

\begin{document}
\begin{abstract}
Combinatorial trees can be used to represent genealogies of asexual individuals. 
These individuals can be endowed with birth and death times, to obtain a so-called `chronological tree'.
In this work, we are interested in the continuum analogue of chronological trees in the setting of real trees. 
This leads us to consider totally ordered and measured trees, abbreviated as TOM trees. 
First, we define an adequate space of TOM trees and prove that under some mild conditions, every compact TOM tree can be represented in a unique way by a so-called contour function, which is right-continuous, admits limits from the left and has non-negative jumps. The appropriate notion of contour function is also studied in the case of locally compact TOM trees. 
Then we study the splitting property of (measures on) TOM trees which extends the notion of `splitting tree' studied in \cite{MR2599603}, where during her lifetime, each individual gives birth at constant rate to independent and identically distributed copies of herself.
We prove that the contour function of a TOM tree satisfying the splitting property is associated to a
spectrally positive L\'evy process that is not a subordinator, both 
in the critical and subcritical cases  of compact trees as well as in the supercritical case of locally compact trees. 
\end{abstract}
\maketitle
\tableofcontents
\section{Introduction}
\label{IntroductionSection}
\subsection{Motivation}
\label{motivationSubsection}
Consider the following population dynamics.
\begin{model}
\label{finiteSplittingTreesModel}
Individuals have i.i.d. lifetimes in $(0,\infty]$ during which they give birth at times of independent Poisson point processes with the same intensity to independent copies of themselves, giving rise to exactly one offspring at each birth event. 
\end{model}
The global history of the population can be encoded by a so-called `chronological tree' as depicted in Figure \ref{DiscreteChronologicalTree}(D). The vertical segments represent the individuals together with their lifetimes, and the vertical axis represents time flowing upwards. 
Hence, the lower endpoint of each vertical segment represents the birth-time while the upper endpoint represents death-time. 
The ancestors of segments can be found by following the dotted lines at the bottom of each segment. 
Following segments downwards and dotted lines to the left gives us the ancestral lines of the tree. 
In \cite{MR2599603} a generalization of the preceding model is proposed and called a `splitting tree', 
following the terminology introduced in \cite{MR1601713}. 
\cite{MR2599603} proposes to encode the tree by its contour, obtained by using a total order which is increasing when descending along ancestral lines; we then traverse the tree at unit speed respecting the total order and recording the height to obtain a piecewise linear function with jumps such as the one depicted in Figure 
\ref{DiscreteChronologicalTree}(A). 
Conversely, starting from a non-negative piecewise linear function with positive jumps and negative slopes, we can interpret it as the contour of a tree as follows. 
Jumps of the function correspond to individuals  and the size of the jumps are interpreted as lifetimes (which we explore from the top of the jump).  
The bottom of each jump is joined backwards to the first vertical line (or jump) it encounters  (by joining along the dotted lines) which then becomes its parent while the height at which it is joined is the birth-time. 
This produces a notion of genealogy between the jumps. 
This interpretation of a function as giving rise to a tree was introduced and explored in \cite{MR1617047} and interpreted via  last in first out (LIFO) queues with a single server.

The process counting the number of individuals alive as time varies in Model \ref{finiteSplittingTreesModel} is known as (binary, homogeneous) Crump-Mode-Jagers (CMJ) process. 
This process is Markovian only in the case of exponentially distributed (or a.s. infinite) lifetimes. 
On the other hand, the numbers of individuals in each successive generation of the genealogy evolve as a Galton-Watson process.  
However, the real surprise of \cite{MR2599603} is that the contours of splitting trees are Markovian: they are L\'evy processes in the subcritical case and otherwise the splitting tree truncated at a given time has a reflected L\'evy process as its contour. 
For descriptions of the contour of non-binary and non-homogeneous versions of Model \ref{finiteSplittingTreesModel}, we refer to \cite{1506.03192}. 

The chronologies of \cite{MR2599603} are basically discrete, in the sense that their associated genealogy is a discrete (or combinatorial) tree from which the chronological tree is obtained by specifying the birth and death times of individuals. 
In short, we wish to extend these results to real trees as follows. 
\begin{enumerate}
\item\label{NotionOfCRTAim} To generalize the notion of chronological tree based on the formalism of real trees; 
these will be the TOM trees alluded to in the abstract, 
\item\label{CodingOfCRTAim} To show that these TOM trees also give rise to a contour which characterizes the tree,
\item\label{SplittingAim} To introduce the splitting property for measures on TOM trees,
\item\label{ContourAim} To identify the law of contours under measures with the splitting property in terms of reflected L\'evy processes.
\end{enumerate}
\begin{figure}
\begin{center}
\subfloat[][]{\includegraphics[width=.4\textwidth]{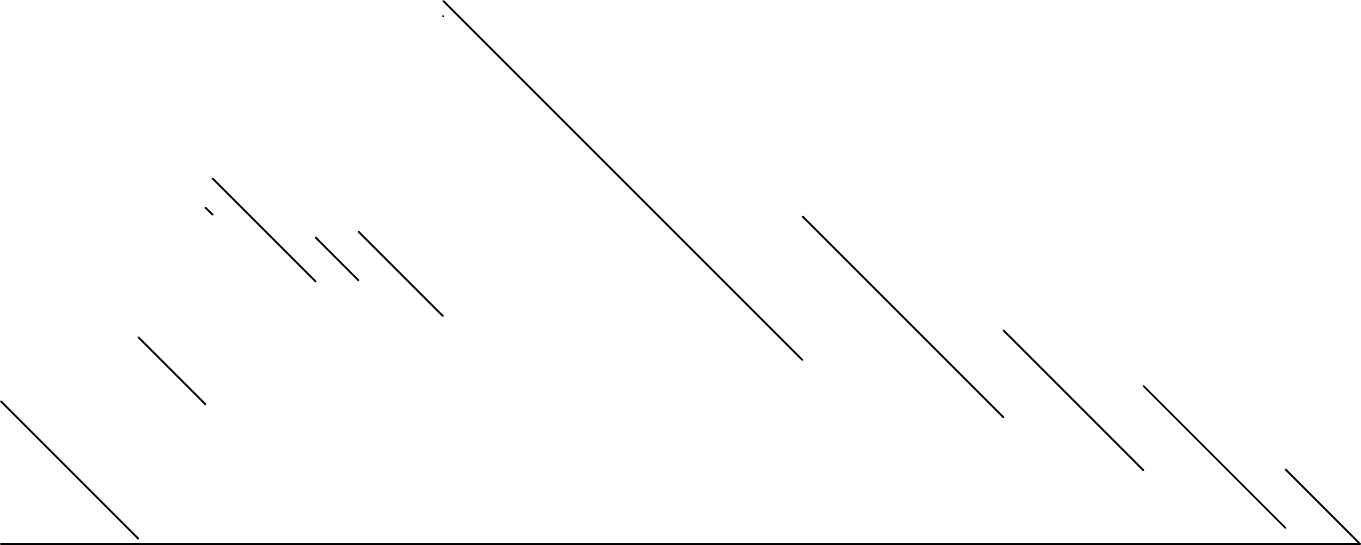}}
\quad
\subfloat[][]{\includegraphics[width=.4\textwidth]{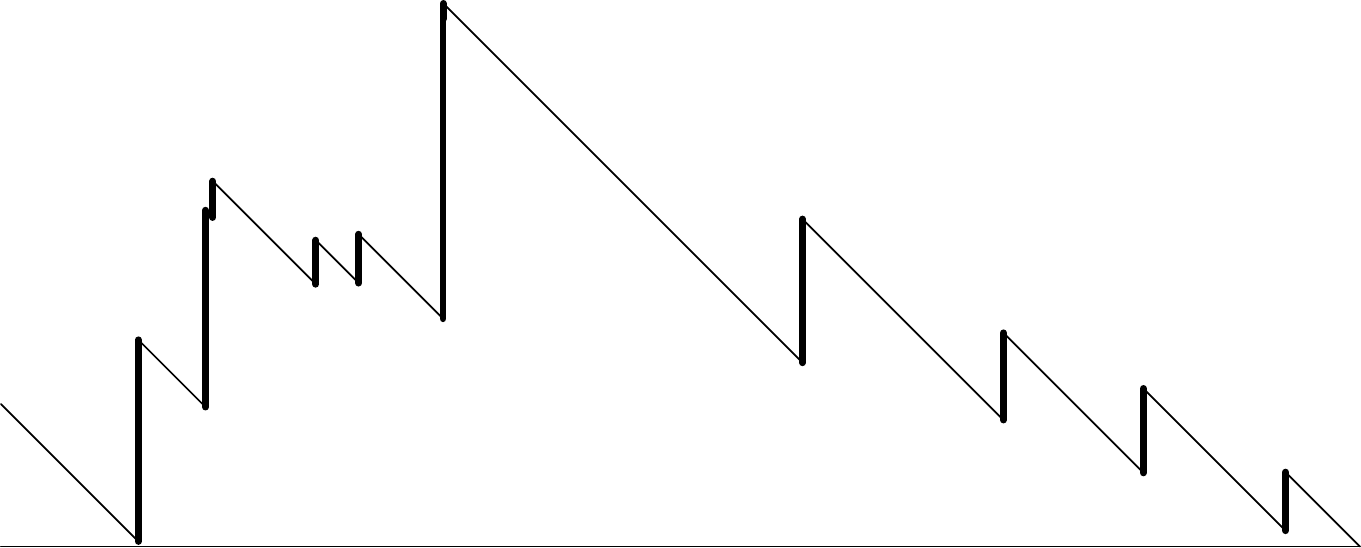}}\\
\subfloat[][]{\includegraphics[width=.4\textwidth]{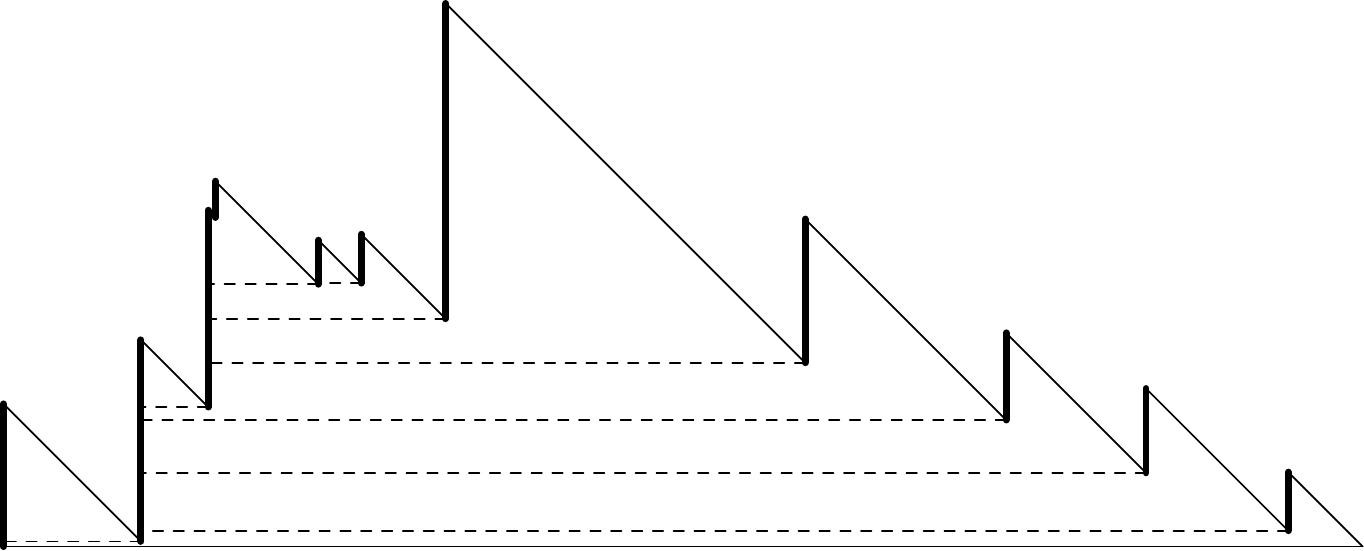}}\quad
\subfloat[][]{\includegraphics[width=.4\textwidth]{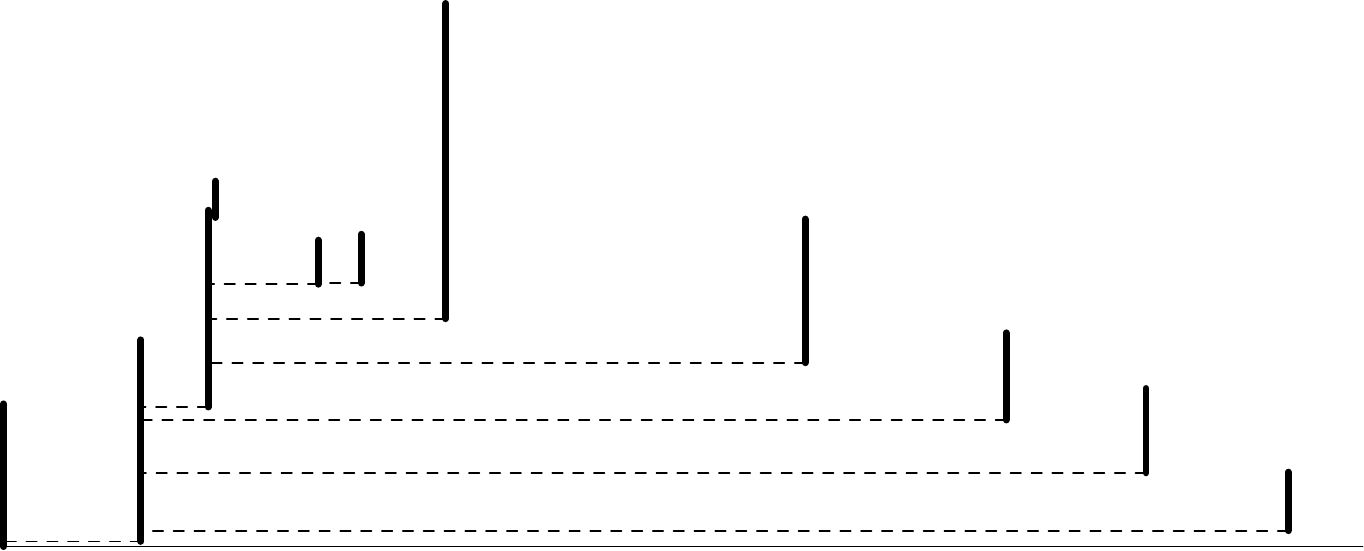}\label{treeModel1}}
\end{center}
\caption{The coding function of a chronological tree with finite length 
where lifetimes are traversed at unit speed: how to recover the tree from the contour. A) Start with a \cadlag\ map with compact support; B) Draw vertical solid lines in the place of jumps; C) Report horizontal dashes lines from each edge bottom left to the rightmost solid point; D) erase diagonal lines.}
\label{DiscreteChronologicalTree}
\end{figure}
Items \ref{NotionOfCRTAim} and \ref{CodingOfCRTAim} are based on modifying the formalism of structured real trees of \cite{Duquesne:fk} to suit our present needs. 
On the other hand, items \ref{SplittingAim} 
 and  \ref{ContourAim}
  are motivated by \cite{MR2599603}, although our arguments differ. 
In forthcoming work, 
we use the above measures with the splitting property to give a construction of supercritical L\'evy trees (which have been constructed as limits of Galton-Watson trees consistent under Bernoulli leaf percolation in \cite{MR2322700} or by relating them to subcritical L'evy trees via Girsanov's theorem in \cite{MR2437534} and \cite{MR2962090}) by applying the height process  of \cite{MR1954248} to the contours constructed in this paper. 
This gives us access to the chronology driving the genealogy encoded in supercritical L\'evy trees and lets us identify the prolific individuals (those with an infinite line of descent) introduced in \cite{MR2455180}. 
Finally, we believe that our construction of supercritical L\'evy trees will provide a snake construction of  supercritical superprocesses (with spatially independent branching mechanisms,  as in \cite{MR1714707} for (sub)critical cases)  and give an interpretation for the backbone decomposition of \cite{MR2794978}.

\subsection{Statement of the results}
The central notion of this work, that of a TOM tree, is based on the metric spaces called real trees.
The definition of real trees mimics the concept of a combinatorial tree defined as a combinatorial graph which is connected and has no cycles. 
\begin{definition}[From \cite{MR1397084} and \cite{MR2221786}]
An \defin{\realtree}\ (or \defin{real tree}) is a metric space $\paren{\tau,d}$ satisfying the following properties:
\begin{description}
\item[Completeness] $\paren{\tau,d}$ is complete.
\item[Uniqueness of geodesics]  For all $\sigma_1,\sigma_2\in\tau$ there exists a unique isometric embedding\begin{linenomath}\begin{esn}
\fun{\phi_{\sigma_1,\sigma_2}}{[0,\imf{d}{\sigma_1,\sigma_2}]}{\tau}
\end{esn}\end{linenomath}such that $\imf{\phi}{0}=\sigma_1$ and $\imf{\phi}{\imf{d}{\sigma_1,\sigma_2}}=\sigma_2$.
\item[Lack of loops] For every injective continuous mapping $\fun{\phi}{[0,1]}{\tau}$ such that $\imf{\phi}{0}=\sigma_1$ and $\imf{\phi}{1}=\sigma_2$, the image of $[0,1]$ under $\phi$ equals the image of $[0,\imf{d}{\sigma_1,\sigma_2}]$ under $\phi_{\sigma_1,\sigma_2}$. 
\end{description}
A triple $\paren{\tau,d,\rho}$ consisting of a real tree $\paren{\tau,d}$ and a distinguished element $\rho\in\tau$ is called a \defin{rooted (real) tree}.
\end{definition}

If $\paren{\tau,d,\rho}$ is a rooted real tree and $\sigma_1,\sigma_2\in\tau$, we define the \defin{closed interval} $[\sigma_1,\sigma_2]$ to be the image of $[0,\imf{d}{\sigma_1,\sigma_2}]$ under $\phi_{\sigma_1,\sigma_2}$; the \defin{open interval}, obtained by removing $\sigma_1$ and $\sigma_2$, is denoted $(\sigma_1,\sigma_2)$. 
We can now define the \defin{genealogical partial order} $\preceq$ by stating that\begin{linenomath}\begin{esn}
\sigma_1\preceq\sigma_2\quad\text{if and only if}\quad \sigma_1\in [\rho,\sigma_2].
\end{esn}\end{linenomath}When $\sigma_1\preceq\sigma_2$ and $\sigma_1\neq\sigma_2$ we write $\sigma_1\prec\sigma_2$. 
Since a tree has no loops, there is a unique element, called the \defin{most recent common ancestor} of $\sigma_1$ and $\sigma_2$ \label{greaterCommonAncestorDefinition}  and denoted\begin{linenomath}\begin{esn}
\sigma_1\wedge \sigma_2, 
\end{esn}\end{linenomath}such that\begin{linenomath}\begin{esn}
[\rho,\sigma_1]\cap[\rho,\sigma_2]=[\rho,\sigma_1\wedge \sigma_2].
\end{esn}\end{linenomath}

The real tree coded by a function has been introduced for continuous functions in \cite{MR2225746} and for c\`agl\`ad (left-continuous with right limits) functions with negative jumps in \cite{Duquesne:fk}. 
We now recall its construction in the setting of \cadlag\ functions, that is right-continuous functions which admits limits from the left, with non-negative jumps. 
Given a \cadlag\ function $\fun{f}{[0,m]}{[0,\infty)}$ with non-negative jumps and such that $\imf{f}{m}=0$ we can define a compact real tree $\paren{\tau_f,d_f,\rho_f}$ as follows. 
Consider the pseudo-distance $d_f$ on $[0,m]$ given by\begin{linenomath}\begin{esn}
\imf{d_f}{t_1,t_2}=\imf{f}{t_1}+\imf{f}{t_2}-2\imf{m_f}{t_1,t_2}\quad\text{where} \quad \imf{m_f}{t_1,t_2}= \inf_{t\in [t_1,t_2]} \imf{f}{t},
\end{esn}\end{linenomath}as well as the corresponding equivalence relationship $\sim_f$ given by\begin{linenomath}\begin{esn}
t_1 \sim_f t_2\text{ if and only if }\imf{d_f}{t_1,t_2}=0.
\end{esn}\end{linenomath}Let $[s]_f$ denote the corresponding equivalence class of $s\in [0,m]$. 
It we equip the quotient space $[0,m]/\sim_f=\set{[t]_f:t\in [0,m]}$ with the induced distance (also denoted $d_f$) and root it at the equivalence class of $m$ (denoted $\rho_f$), we obtain a compact rooted real tree. 
(The proof is similar to the corresponding statement for c\`agl\`ad functions in Lemma 2.1 of \cite{Duquesne:fk}.) 
As an example, note that if $[s_i]_f=\sigma_i$ for $i=1,2$ and $s\in [s_1,s_2]$ is such that $\imf{m_f}{s_1,s_2}=\imf{f}{s}$ 
then $[s]_f=[s_1]_f \wedge [s_2]_f$. 
Because of this, the interval $[[s]_f,\rho_f]$ can be identified with the set of $t\in [s,m]$ such that $\imf{f}{t}=\imf{m_f}{s,t}$. 

We now adopt (and modify) the insight of \cite{Duquesne:fk} which is to notice that the function $f$ endows the real tree $\tau_f$ with additional structure: a total order $\leq$ where two equivalence classes $\sigma_1,\sigma_2 \in\tau_f$ satisfy $\sigma_1\leq \sigma_2$ if and only if $\sup\sigma_1\leq\sup \sigma_2$, and a measure $\mu$ equal to the push-forward of Lebesgue measure under the canonical projection, {\label{ProjectionFromIntervalToRealTreeDefinition} denoted $p_f$,} from $[0, m]$ to $\tau_f$. 
The abstraction of this situation (i.e., not defining the tree from any given $f$) lies at the heart of the notion of TOM trees.

\begin{definition}
A \defin{real tree} $\paren{\tau,d,\rho}$ is called \defin{totally ordered} if there exists a total order $\leq$ on $\tau$ which satisfies
\begin{description}
\item[Or1] $\sigma_1\preceq \sigma_2$ implies $\sigma_2\leq \sigma_1$ and
\item[Or2] $\sigma_1<\sigma_2$ implies $[\sigma_1,\sigma_1\wedge\sigma_2)<\sigma_2$.
\end{description}A totally ordered real tree is called \defin{measured} if there exists a 
measure $\mu$ on the Borel sets of $\tau$ satisfying:
\begin{description}
\item[Mes1] 
$\mu$ is locally finite and for every $\sigma_1<\sigma_2$ we have that
\begin{linenomath}
\begin{esn}
\imf{\mu}{\set{\sigma:\sigma_1\leq \sigma\leq \sigma_2}}>0. 
\end{esn}
\end{linenomath}
\item[Mes2] $\mu$ is diffuse. 
\end{description}A totally ordered measured tree, typically denoted $\tr{c}=\paren{\paren{\tau,d,\rho},\leq,\mu}$, will be referred to as a \defin{TOM tree}. 
We will say that $\tau$ (or $\paren{\tau,d,\rho}$) is the tree part of $\tr{c}$. 
\end{definition}
\label{TOMtreedefinitionpage}
We will be exclusively interested in locally compact TOM trees. 
In this case, the measure $\mu$ is actually finite when the tree is compact and hence $\sigma$-finite otherwise. 
As defined, TOM trees are intimately linked with the structured trees of \cite{Duquesne:fk}. 
The main difference between TOM trees and the structured trees of  \cite{Duquesne:fk} is that the total order $\leq$ is in a sense the opposite of ours, since it is an extension of the genealogical partial order. 
Extending instead to the inverse genealogical partial order is important in order to obtain a coding function which is \cadlag\ and has non-negative jumps, since we will be interested in random TOM trees whose coding function is a Markov process. 
Indeed, to get a structured tree, it suffices to define a new total order $\leq^\text{st}$ so that $s\leq^\text{st} t$ if $t\leq s$. 
Then the tree $\paren{\paren{\tau,d,\rho},\leq^\text{st}, \mu}$ is a structured tree. 
When coding a tree by a \cadlag\ function $f$, it suffices to consider the function $t\mapsto f(m-t)$ to be in the setting of \cite{Duquesne:fk}. 
We make the additional assumption of diffusivity of $\mu$ to obtain what we think is a slightly easier definition of a coding function for TOM trees.

 With our definition, 
not every \cadlag\ function $\fun{f}{[0,m]}{[0,\infty)}$ with non-negative jumps and $\imf{f}{m}=0$ codes a compact TOM tree. 
We have to assume additionally  that the equivalence classes $f$ gives rise to have zero Lebesgue measure, 
so that the pushforward measure is non-atomic. 
One way to ensure this is to impose that the equivalence classes generated by $f$ are at most countable.
This assumption is satisfied with probability 1 if $f$ is the sample path of a L\'evy process with no negative jumps that is not a subordinator (cf. Proposition \ref{levyProcessCodesTreeProposition} in Section \ref{CompactSplittingTreeSection}). 

There is a simpler way to define a total order satisfying \defin{Or} when the underlying tree is binary. 
A real tree $\tau$ is said \defin{binary} when $\tau\setminus\set{\sigma}$ has at most 3 connected components for every $\sigma\in\tau$. 
In that case, $\sigma$ is called a branching point if $\tau\setminus\set{\sigma}$ has 3 connected components, the one containing $\rho$ and two others. It can be easily seen that these three connected components are also trees.  
Suppose that for every branching point $\sigma$ we are given an orientation, which declares, for the two components of $\tau\setminus\set{\sigma}$ which do not contain $\rho$, which one is the `left' subtree rooted at $\sigma$ and which one is the `right' subtree rooted at $\sigma$. Then we can define a total order $\leq$ on $\tau$ as follows. For any $\sigma_1,\sigma_2\in\tau$, $\sigma_1\preceq \sigma_2$ implies $\sigma_2\leq \sigma_1$ (so that $\leq$ satisfies \defin{Or1}). 
Otherwise, note that $\sigma_1\wedge \sigma_2$ is a branching point and we then define\begin{linenomath}\begin{esn}
\sigma_1\leq \sigma_2\quad  \text{if  and only if}\quad 
\sigma_1
\text{ belongs to the left component of }\tau\setminus\set{\sigma_1\wedge \sigma_2}
.
\end{esn}\end{linenomath}Then $\leq$ is a total order on $\tau$ which satisfies \defin{Or}. 

The importance of the notion of TOM trees is that they allow the construction of a contour which codes the tree in the aforementioned sense. 
To formalize this, we will say that two compact TOM trees are \defin{isomorphic} if there is an isometry from one to the other which respects the total order and maps the associated diffuse measures one onto the other. 
\begin{theorem}
\label{CRTCodingTheorem}
Let $\tr{c}$ be a compact TOM tree and let $m=\imf{\mu}{\tau}$. 
There exists a \cadlag\ function $\fun{f_{\tr{c}}}{[0,m]}{[0,\infty)}$ with no negative jumps such that the tree coded by $f_{\tr{c}}$ is isomorphic to $\tr{c}$. 
In particular, if $\tr{c}$ and $\tr{\tilde c}$ are non-isomorphic TOM trees then $f_{\tr{c}}\neq f_{\tr{\tilde c}}$.
\end{theorem}
The function $f_{\tr{c}}$ is called the \defin{contour process} for reasons that will be clear upon its construction. 
Also, note that the above correspondence gives us a canonical way to explore the tree: at time $t$ we visit the equivalence class $[t]$ of $t$ under $f_{\tr{c}}$. 
The mapping $\fund{\phi}{t}{[t]}$ is called the \defin{exploration process}. 
We can then interpret the measure $\mu$ on the compact TOM tree as a measure of time. 
Indeed, if we traverse the interval $[0,m]$ from left to right at unit speed, the exploration process gives us a way to traverse the tree respecting  the order $\leq$ (except at a countable number of exceptional points) and $\imf{\mu}{V}$ is the time it takes to explore the set $V\subset \tau$. 

Theorem \ref{CRTCodingTheorem} follows from Theorem 1.1 in \cite{Duquesne:fk}, where the functions $f_{\tr{c}}$ are given a characterization (as those satisfying a certain property \defin{Min}). For completeness, we give a proof of Theorem \ref{CRTCodingTheorem} in Appendix \ref{CRTCodingTheoremProof}. 

In order to relate the compact and locally compact cases, we introduce the truncation of a locally compact TOM tree. 
Let $r>0$ and $\tr{c}=\paren{\paren{\tau,d,\rho},\leq,\mu}$ be a locally compact TOM tree. 
\begin{definition}
The \defin{truncation} of $\tr{c}$ at level $r$, is the TOM tree $\tr{c}^r=\paren{\paren{\tau^r,d^r,\rho^r},\leq^r,\mu^r}$  where\begin{linenomath}\begin{esn}
\tau^r=\set{\sigma\in\tau:\imf{d}{\sigma,\rho}\leq r},
\end{esn}\end{linenomath}$d^r$ is the restriction of $d$ to $\tau^r\times \tau^r$, $\rho^r=\rho$, $\leq^r$ is the restriction of $\leq$ to $\tau^r\times \tau^r$ and  $\imf{\mu^r}{A}=\imf{\mu}{A\cap \tau^r}$. 
\end{definition}
It is simple to see that $\paren{\tau^r,d^r,\rho^r}$ is a real tree. 
As noted in \cite{MR2322700}, the closed balls of a locally compact real tree are compact by the Hopf-Rinow theorem (cf. \cite{MR2307192} or \cite{MR1835418}); this is where the completeness assumption in our definition of real trees comes into play since, without it, $[0,1)$ would be a locally compact tree without compact truncations at level $r\geq 1$. 
Hence, the truncation $\tr{c}^r$ at level $r$ of a locally compact TOM tree $\tr{c}$ is a compact TOM tree and can therefore be coded by a function $f^r$. 
The functions $f^r$ are related by a time-change in Proposition \ref{HeightProcessOfTruncatedChronologicalTreeProposition} of Section \ref{LocallyCompactCRTSection}.


Let $\cts_c$ stand for the set of equivalence classes of isomorphic compact TOM trees (cf. Remark 7.2.5 of \cite[p. 251]{MR1835418} for why this class actually constitutes a set, although this is also a consequence of Theorem \ref{CRTCodingTheorem}). 
We define $\fun{X_t}{\cts_c}{[0,\infty)\cup\set{\dagger}}$ as follows, where $\dagger$ is a so-called cemetery state. 
Given (a representative) $\tr{c}=\paren{\paren{\tau,d,\rho},\leq,\mu}\in\cts_c$, let $f$ be the contour process of $\tr{c}$ and\begin{linenomath}\begin{esn}
\imf{X_t}{\tr{c}}=\begin{cases}
\imf{f}{t}&t\leq\imf{\mu}{\tau}\\
\dagger&\text{otherwise}
\end{cases}
\end{esn}\end{linenomath} 
We then endow $\cts_c$ with the \sa\ $\mc{C}=\sag{X_s:s\geq 0}$.

We now turn to the  TOM trees coded by L\'evy processes. 
We consider the canonical process on two functional spaces: the usual Skorohod space $\mathbf{D}$ of \cadlag\ trajectories $\fun{f}{[0,\infty)}{[0,\infty)}$ (as introduced for example in \cite[Ch 19, p. 380]{MR1876169}) and the associated excursion space\begin{linenomath}\begin{esn}
\mathbf{E}=\set{\fun{f}{[0,\infty)}{[0,\infty)\cup\set{\cemetery}}:f\text{ is \cadlag,  }\imf{f}{t}=\dagger \text{ iff }t\ge\zeta\text{ for some }\zeta>0\text{ and }\imf{f}{\zeta-}=0}.
\end{esn}\end{linenomath}Note that the cemetery state $\cemetery$ is isolated  and absorbing and $\zeta=\imf{\zeta}{f}$ is termed the lifetime of the excursion $f$. 

On both functional spaces $\mathbf{D}$ and $\mathbf{E}$ we abuse our previous notation and define the canonical process $X$ by\begin{linenomath}\begin{esn}
\imf{X_t}{f}=\imf{f}{t}
\end{esn}\end{linenomath}as well as the \sa s $\mc{D}$ and $\mc{E}$ on $\mathbf{D}$ and $\mathbf{E}$ respectively as $\sag{X_t:t\geq 0}$. 
Note that in view of the correspondence between functions and trees, if $\tr{c}=\paren{\paren{\tau,d,\rho},\leq,\mu}$ is a compact TOM tree, we can define $\zeta=\imf{\zeta}{\tr{c}}=\imf{\mu}{\tau}$. 
We also introduce the canonical filtration $\paren{\F_t,t\geq 0}$ where\begin{linenomath}\begin{esn}
\F_t=\sag{X_s:s\leq t}.
\end{esn}\end{linenomath}

We now consider spectrally positive (i.e., with no negative jumps) L\'evy processes. 
The reader can consult 
\cite{MR1406564} (especially Chapter VII) for the adequate background. 
Let $\Psi$ be the Laplace exponent of a spectrally positive L\'evy process which is not a subordinator (equivalent to assuming $\imf{\Psi}{\infty}=\infty$). 
This is a convex function written as
\begin{linenomath}\begin{esn}
\imf{\Psi}{\lambda}=-\kappa+\alpha \lambda+\beta\lambda^2+\int_0^\infty \bra{e^{-\lambda x}-1+\lambda x\indi{x\leq 1}}\, \imf{\pi}{dx}
\end{esn}\end{linenomath}where $\kappa,\beta\geq 0$, $\alpha\in\re$ and $\pi$ is the L\'evy measure of $\Psi$, which satisfies $\int_0^\infty 1\wedge x^2\, \imf{\pi}{dx}<\infty$. 
Let $\p_x$ (or $\p^\Psi_x$ when $\Psi$ is not clear from the context) be the law of a spectrally positive L\'evy process with Laplace exponent $\Psi$ which starts at $x$. 
When $\kappa >0$, the $\Psi$-L\'evy process has the same law as the L\'evy process with Laplace exponent $\psi +\kappa$ killed after an independent exponential time with parameter $\kappa$; 
we will assume that a killed L\'evy process jumps to the cemetery state $\dagger$ upon killing. 
We say that $X$ is (sub)critical if $X$ does not drift or jump to $+\infty$ (or equivalently $\imf{\Psi'}{0}\geq 0$ and $\kappa=0$) and otherwise it will be termed supercritical. 
Let $b$ be the largest root of the convex function $\Psi$, so that $b>0$ if and only if $\Psi$ is supercritical. 
If $\underline X$ stands for the cumulative minimum process given by\begin{linenomath}\begin{esn}
\underline X_t=\inf_{s\leq t} X_s,
\end{esn}\end{linenomath}then the reflected process $X-\underline X$ is a strong Markov process (cf. \cite[Ch. VI]{MR1406564}) and $\underline X$ is a version of its local time at zero. 
Given a supercritical exponent $\Psi$ 
we can define an associated subcritical exponent $\imf{\Psi^\#}{\lambda}=\imf{\Psi}{\lambda+b}$ 
which corresponds to conditioning a $\Psi$-L\'evy process on reaching arbitrarily low levels, 
in the sense that its law equals\begin{linenomath}\begin{esn}
\lim_{a\to-\infty}\imf{\p_0}{\cond{\cdot}{\underline X_\infty<a}}
\end{esn}\end{linenomath}on every $\F_s$. 
(Cf. Lemma 7 in \cite[Ch. VII]{MR1406564} and Lemme 1 in \cite{MR1141246}.) 
We let $\nu$ stand for the excursion measure of $X-\underline X$ away from zero;
the measure $\nu$ may charge excursions of infinite length. 
Indeed, if $\p^\rightarrow$ denotes the law of the L\'evy process after it approaches its overall minimum from the left for the last time  
then\begin{linenomath}\begin{esn}
\nu=\nu^\#+b\p^\rightarrow,
\end{esn}\end{linenomath}where $\nu^\#$ is the excursion measure associated to $\Psi^\#$ which is concentrated on excursions with finite length. 



Also, let us define the law $\q_x$ on $E$ as the law of the L\'evy process killed (and sent to $\dagger$) upon reaching zero under $\p_x$. 
Then, under $\nu$, the canonical process is strongly Markovian with the same semigroup as $\q_x$. 

Define\begin{linenomath}\begin{esn}
\overline X_t^r= \paren{\max_{s\leq t}X_s-r}^+
\end{esn}\end{linenomath}and let $X^r$ be the process obtained by stopping $X-\overline X^r$ when it reaches zero. 
(In the killed case, $X-\overline X^r$ might be killed before $X-\overline X^r$ reaches zero; in this case, we simply concatenate independent copies started at $r$ until a copy reaches zero). 
We let $\nu^r $ be the push-forward of $\nu$ by  $X^r$. 
We also consider the law $\eta^r$ of the tree coded by the canonical process under $\nu^r$. Note that the measures $\eta^r$ are concentrated on binary trees.
Our next result shows that we can define a measure $\eta^\Psi$ whose image under truncation at level $r$ equals $\eta^r$ and that the measure $\eta^\Psi$ is characterized by a self-similarity property termed the splitting property.

To define such a property, suppose that $\tr{c}=\paren{\paren{\tau,d,\rho},\leq,\mu}$ is a TOM tree. 
Truncate $\tr{c}$ at height $r$ to obtain $\tr{c}^r$ and suppose that $\tr{c}^r$ has a total measure greater than $t$. 
Let $\phi$ stand for the exploration process of $\tr{c}^r$. 
Then, on the interval $[\rho,\imf{\phi}{t})$ we can define left subtrees and right subtrees. 
The right subtrees are characterized by the existence of $\sigma\in [\rho,\imf{\phi}{t})$ such that\begin{linenomath}\begin{esn}
\set{\tilde\sigma: [\imf{\phi}{t},\sigma)<\tilde\sigma\leq\sigma}\neq\set{\sigma}.
\end{esn}\end{linenomath}The right subtree at $\sigma$ is then equal to\begin{linenomath}\begin{esn}
R_{t,\sigma}=
\set{\tilde\sigma: [\imf{\phi}{t},\sigma)<\tilde\sigma\leq\sigma}
\end{esn}\end{linenomath}and we can collect all these right subtrees in the measure\begin{linenomath}\begin{esn}
\Xi_t^r=\sum_{R_{t,\sigma}\neq\set{\sigma}}\delta_{\paren{\imf{d}{\rho,\sigma},R_{t,\sigma}}}.
\end{esn}\end{linenomath}
In the following definition, recall that $X(\tr{c})$ is the contour associated to the tree $\tr{c}$.
\begin{definition}
A measure $\kappa$ on locally compact TOM trees satisfies the \defin{splitting property} if for any $r>0$, on defining $\kappa^r$ as the image of $\kappa$ under truncation at height $r$, 
\begin{enumerate}
\item\label{FinitudeInSplittingPropertyDefinition} for any $t>0$, $\imf{\kappa^r}{\zeta>t}<\infty$ and
\item\label{CompactSplittingPropertyDefinition} on the set $\set{\zeta>t}$ and conditionally on $X_t=x$, $\Xi_t^r$ is a Poisson random measure on $[0,x]\times \cts_c$ with intensity $\imf{\leb}{ds}\otimes \imf{\kappa^{r-s}}{d\tr{c}}$: 
if 
$\fun{h}{[0,\infty)\cup\set{\dagger}}{[0,\infty)}$ is measurable and vanishes at $\dagger$ 
and 
$\fun{g}{[0,\infty)\times \cts_c}{[0,\infty)}$ is measurable then\begin{linenomath}\begin{esn}
\imf{\kappa^r}{\imf{h}{X_t}e^{-\Xi_t^r g}}=\imf{\kappa^r}{\imf{h}{X_t}e^{-\int_{[0,X_t]\times \cts_c} \imf{g}{s,\tr{c}}\,\imf{\kappa^{r-s}}{d\tr{c}}\,ds}}.
\end{esn}\end{linenomath}
\end{enumerate}
\end{definition}
Note that no assumption is made concerning conditional independence of $\Xi_t^r$ and $\sag{X_s,s\leq t}$ given $X_t$. 
\begin{remark}
As consequences of the definition, we will see that 
when $\kappa$ is concentrated on compact TOM trees then 
the splitting property admits the non-truncated version
\begin{linenomath}\begin{esn}
\label{compactVersionOfSplittingProperty}
\imf{\kappa}{\imf{h}{X_t}e^{-\Xi_t^\infty g}}=\imf{\kappa}{\imf{h}{X_t}e^{-\int_{[0,X_t]\times \cts_c} \imf{g}{s,\tr{c}}\,\imf{\kappa}{d\tr{c}}\,ds}}
\end{esn}\end{linenomath}and implies the integrability properties
\begin{linenomath}\begin{esn}\imf{\kappa}{\zeta>t}<\infty
\quad\text{and}\quad
\int_0^\infty 1\wedge \imf{\zeta}{\tr{c}}\, \imf{\kappa}{d\tr{c}}<\infty.
\end{esn}\end{linenomath}%
See the proof of the forthcoming Theorem \ref{MeasuresWithSplittingPropertyCompactCaseTheorem}. 
\end{remark}

It turns out that a class of measures on binary TOM trees with the splitting property are in correspondence with L\'evy processes, under an additional assumption of constant sojourn,  as the next result shows. 

The splitting trees of Model \ref{finiteSplittingTreesModel} previously introduced have the property that they treat elements of the tree in an equal manner, as in the following definition.
\begin{definition}
A  TOM tree $\tr{c}=\paren{\paren{\tau,d,\rho},\leq,\mu}$ has \defin{sojourn $\boldsymbol{a} \geq 0$} if $\imf{\mu}{[\rho,\sigma]}=a \,\imf{d}{\rho,\sigma}$ for every $\sigma\in\tau$.
\end{definition}
In other words, a TOM tree has sojourn $a$ if the measure $\mu$ equals $a$ times Lebesgue measure (on the tree). 
It is interesting that the sojourn $a$ can be $0$ in the above definition even if $\mu$ is a non-trivial measure. 
This happens in particular for the trees coded by Brownian excursions since Brownian motion spends zero time (with respect to Lebesgue measure) at its cumulative minimum process. 

Recall the measures $\eta^r$ on binary trees defined from the (reflected) L\'evy process with Laplace exponent $\Psi$.
\begin{theorem}
\label{MeasuresWithSplittingPropertyCompactCaseTheorem}
%
There exists a unique measure $\eta^\Psi$ on locally compact TOM trees whose truncation at level $r$ equals $\eta^r$. 
The measure $\eta^\Psi$ satisfies the splitting property, is concentrated on binary trees, 
has constant sojourn\begin{linenomath}\begin{esn}
a=\lim_{\lambda\to\infty}\frac{\lambda}{\imf{\Psi}{\lambda}},
\end{esn}\end{linenomath}and assigns finite measure to non-compact TOM trees. 
When $\Psi$ is (sub)critical, $\eta^\Psi$ is the push-forward of $\nu$ under the  mapping taking functions to trees.  $\eta^\Psi$ charges non-compact trees if and only if $\Psi$ is supercritical. 

Conversely, if a non-zero measure $\kappa$ on locally compact TOM trees has the splitting property, is concentrated on binary trees and there exists $a\geq 0$ such that under $\kappa$ the tree has sojourn $a$ almost everywhere, then there exists a spectrally positive L\'evy process with Laplace exponent $\Psi$ such that $\kappa=\eta^\Psi$. 
\end{theorem}
The assumption that the measure $\kappa$ has sojourn $a$ is necessary to be able to code the tree by a L\'evy process as the next example shows. 
Consider Model \ref{finiteSplittingTreesModel}. 
Lambert shows in Theorem 4.1 and Remark 2 of \cite[p 373]{MR2599603} that if $b$ times the expected lifetime is less than or equal to $1$ then the resulting real tree is compact with probability one. 
In this setting, one can use Lebesgue measure on each of the (finite number of) segments comprising the tree, thus obtaining a tree of sojourn $1$, and the contour process of the compact TOM tree is a compound Poisson process minus a drift of slope $1$ until it reaches zero. 
However, one might use an additional speed distribution $S$ on $(0,\infty)$ and associate to individuals iid speeds with distribution $S$ at which their lifetimes will be traversed (or equivalently, defining the multiple of Lebesgue measure which will be used on each interval in the tree representing the lifetime). 
The compact TOM tree we obtain has the splitting property but does not have sojourn $a$ and is not coded by a L\'evy process (since there is no unique drift). 
However, there is a (random) time-change which (by giving all individuals the same non-random speed) takes the coding process into a L\'evy process and this begs the question on whether there is a characterization of binary trees with the splitting property as time-changes of L\'evy processes. 
This question is left open. 
Another question that is not addressed is on characterizing splitting trees which are not binary.

\subsection{Organization}
The paper is organized as follows. 
In Section \ref{CompactSplittingTreeSection}, we study $\sigma$-finite measures on compact TOM trees and how they are related to L\'evy processes through a proof of Theorem \ref{MeasuresWithSplittingPropertyCompactCaseTheorem} in this case. 
In Section \ref{LocallyCompactCRTSection}, we study the truncation operator on compact real trees and how locally compact TOM trees can be therefore constructed as direct limits of consistent families of compact TOM trees. 
We then give in Section \ref{ReflectedProcessSection} 
some constructions of the reflected L\'evy processes which allow us, in Section \ref{locallyCompactSplittingTreeSection} to construct $\sigma$-finite measures on locally compact TOM trees which have the splitting property and to give the proof of Theorem \ref{MeasuresWithSplittingPropertyCompactCaseTheorem} in full generality. 
Finally, in 
Appendix \ref{CRTCodingTheoremProof}%
, we handle the deterministic aspects of the space of TOM trees culminating in a proof of 
Theorem \ref{CRTCodingTheorem}. 



\section{Measures on compact TOM trees with the splitting property}
\label{CompactSplittingTreeSection}
In this section, we provide the characterization of measures on compact TOM trees with the splitting property stated in Theorem \ref{MeasuresWithSplittingPropertyCompactCaseTheorem}. 
We first show that the canonical process codes a compact real tree under the excursion measure of a (sub)critical L\'evy process. 
\begin{proposition}
\label{levyProcessCodesTreeProposition}
Let $X$ be a (sub)critical spectrally positive L\'evy process with Laplace exponent $\Psi$. 
Let $\nu$ be the intensity measure of excursions of $X$ above its cumulative minimum. 
Then, $\nu$-almost surely, the equivalence classes\begin{linenomath}\begin{esn}
[s]_X=\set{t\in [0,\zeta): X_t=X_s=\underline X_{[s,t]}}
\end{esn}\end{linenomath}have at most 
three 
elements for all $s\in [0,\zeta)$. 
\end{proposition}

\begin{proof}
Under $\nu$, we have $X_t>0$ for all $t\in (0,\zeta)$. 
Hence $[0]_X$ does not have more than 3 elements. 

Suppose there exist $0<t_1<\cdots<t_4$ such that $t_i\in [t_1]_X$. 
Chose $u_i\in (t_i,t_{i+1})\cap\ra$ for $i=1,2,3$ and note that $\underline X_{[u_1,u_2]}=\underline X_{[u_2,u_3]}$. 
Hence
\begin{linenomath}\begin{esn}
\set{\exists t,  \#[t]_X\geq 4}\subset \bigcap_{\eps\in\ra_+}\bigcup_{\eps<u_1<u_2<u_3\in \ra} \set{\underline X_{[u_1,u_2]}=\underline X_{[u_2,u_3]}}. 
\end{esn}\end{linenomath}We now show that the right-hand set has $\nu$-measure zero. 
By countable subadditivity, it suffices to show
\begin{linenomath}\begin{esn}
\imf{\nu}{\underline X_{[u_1,u_2]}=\underline X_{[u_2,u_3]}, u_3<\zeta}=0.
\end{esn}\end{linenomath}for $\eps<u_1<\cdots<u_3$. However, using the Markov property at $u_1$ under $\nu$, we see that\begin{linenomath}\begin{esn}
\imf{\nu}{\underline X_{[u_1,u_2]}=\underline X_{[u_2,u_3]}, u_3<\zeta}
=\imf{\nu}{\indi{\eps<\zeta}\imf{\q_{X_\eps}}{\underline X_{[u_1-\eps,u_2-\eps]}=\underline X_{[u_2-\eps,u_3-\eps]}}}. 
\end{esn}\end{linenomath}Hence, again by the Markov property, it suffices to show that $\imf{\p_x}{\underline X_{u_1}=\underline X_{[u_1,u_2]}}=0$. 
However,
\begin{linenomath}\begin{esn}
\imf{\p_x}{\underline X_{u_1}=\underline X_{[u_1,u_2]}}=\imf{\se_x}{\imf{h}{X_{u_1},\underline X_{u_1}}}
\quad\text{where}\quad 
\imf{h}{y,z}=\imf{\p_y}{\underline X_{u_2-u_1}=z}. 
\end{esn}\end{linenomath}However, since $X$ is spectrally positive, then $0$ is regular for $(-\infty,0)$ (cf. Thm. 1 \cite[Ch. VII]{MR1406564}) and so Lemma 1 in \cite{Pecherskii:1969aa} (or Theorem 1 in both \cite{MR3098676} and \cite{MR2978134}) tell us that the law of $\underline X_{u_2-u_1}$ is non-atomic, which implies that $h$ is actually zero. 
%
\end{proof}
%
We now do an analysis on TOM trees to reduce our proof of the splitting property of $\eta^\Psi$ to a simple property of the excursion measures and to motivate the construction of the L\'evy process out of a measure on trees with the splitting property.

Let $\fun{f}{[0,m]}{\re}$ be a \cadlag\ function which codes a TOM tree $\tr{c}$ and consider $s\in (0,m)$ and the class of $s$ under $\sim_f$ denoted $\sigma=[s]_{f}$. 
We first identify the subtrees attached to the right of $(\sigma,\rho]$ with the excursions of $f$ above its cumulative minimum on $[s^*,m]$, where\begin{linenomath}\begin{esn}
s^*=\inf\set{t\geq s: \imf{f}{t}<\imf{f}{s}}.
\end{esn}\end{linenomath}We have already identified $R_\sigma$ with the image of $[s^*,m]$ by the exploration process in Subsection \ref{ElementaryOperationsOnCCRTsSection}.
 Also, when discussing the tree coded by a function in Section \ref{IntroductionSection}, we identified the interval $[\sigma,\rho]$ with the image of\begin{linenomath}\begin{esn}
\set{t\geq s^*:\imf{f}{t}=\imf{m_f}{s^*,t}}.
\end{esn}\end{linenomath}Hence, $R_\sigma\setminus [\sigma,\rho]$ can be identified with the excursions of $f$ above its cumulative minimum on the interval $[s^*,m]$. %
To be more precise, say that `
$\tilde\sigma\in [\rho,\sigma)$ is a branching point of $R_\sigma$ if $R_\sigma\setminus\set{\tilde\sigma}$ has at least three components. 
One of these components contains $\sigma$ and a different one contains $\rho$. 
If we join all the other components to $\tilde\sigma$ (of which there is only one in the case of binary trees), we obtain a compact TOM tree which can be termed the tree at $\tilde\sigma$ to the right of $[\sigma,\rho]$. 
Let us call it $\tr{c}_{\tilde\sigma,\sigma}$. 
The set of $\tilde\sigma\in [\sigma,\rho]$ for which $\tr{c}_{\tilde\sigma, \sigma}$ is not reduced to $\tilde\sigma$ is at most countable. 
Indeed, by compactness, there can be only finitely many of them with height exceeding a given $\eps>0$. 
Let $\sigma_1,\sigma_2,\ldots$ be the branching points of $R_\sigma$ in $[\sigma,\rho]$. 
Notice that if $\paren{\sigma_i^n,n\in\na}$ is a sequence in $[\sigma,\sigma_i)$ converging to $\sigma_i$ then the real tree of $\tr{c}_{\sigma_i,\sigma}$, denoted $\tau_{\sigma_i,\sigma}$ is given by\begin{linenomath}\begin{esn}
\tau_{\sigma_i,\sigma}=\bigcap_n R_{\sigma_i^n}\cap L_{\sigma_i}. 
\end{esn}\end{linenomath}

Hence, $\tr{c}_{\eta_i,\sigma}$ can be coded by\begin{linenomath}\begin{esn}
f^i=\imf{f}{\cdot+s_i}-\imf{f}{s^i}\text{ on }[0,s^i-s_i]\end{esn}\end{linenomath}
where\begin{linenomath}\begin{esn}
s_i=\lim_{n\to\infty}\imf{\mu}{L_{\eta^n_i}}\quad\text{and}\quad s^i=\imf{\mu}{L_{\eta_i}}. 
\end{esn}\end{linenomath}However, if $x_i=\imf{d}{\eta_i,\rho}$, we also have the representation\begin{linenomath}\begin{esn}
s_i=\inf\set{t\geq s: \imf{f}{t}\leq \imf{f}{s}-x_i}\quad\text{and}\quad s^i=\inf\set{t\geq s: \imf{f}{t}< \imf{f}{s}-x_i}
\end{esn}\end{linenomath}by our previous identification of the elements of $[\sigma,\rho]$, which shows that $\tr{c}_{\eta_i,\sigma}$ can be coded by an excursion of $f$ above its cumulative minimum. 

Under the additional assumptions that $\tr{c}$ is binary with sojourn $a$ and that $s=\imf{\mu}{L_\sigma}$, 
we can reconstruct $\imf{f}{s+\cdot}$ from $\imf{f}{s}$, the sequence $\paren{x_i}$ and each of the functions $f^i$. 
Indeed, for any $l\in [0,\imf{f}{s}]$, let $\sigma_l$ be the point of $[\sigma,\rho]$ at distance $l$ from $\sigma$ and consider $S_l$ to be the measure of $R_{\sigma}\cap L_{\sigma_l}$. 
Then because of constant sojourn, say $a$,
\begin{linenomath}\begin{esn}
S_l=\sum_{\imf{d}{\eta_i,\sigma}\leq l} \imf{\mu}{\tau_{\eta_i,\sigma}}+ a l.
\end{esn}\end{linenomath}
Note that $\imf{d}{\eta_i,\sigma}\leq l$ if and only if $\imf{f}{s}-x_i\leq l$ 
and $\imf{\mu}{\tau_{\eta_i,\sigma}}$ is the Lebesgue measure of the interval of definition of $f^i$, so that $S$ can be constructed from the aforementioned quantities. 
Also,\begin{linenomath}\begin{esn}
S_{x_i-}=s_i\quad\text{and}\quad S_{x_i}=s^i.
\end{esn}\end{linenomath}Now let 
$\Lambda$ be the right-continuous inverse of $S$. 
By considering the dense set of measures of left sets (cf. Lemma \ref{DensityOfImageLemma}), we see that
\begin{linenomath}\begin{esn}
\Lambda_t=\imf{d}{\sigma, \imf{\phi}{t+s}\wedge \sigma}.
\end{esn}\end{linenomath}
for $t\in [0,\imf{\mu}{R_\sigma}]$, where we have seen that $\imf{\mu}{R_\sigma}=\imf{\zeta}{f}-s$. 
Also, if $\imf{\phi}{t+s}$ does not belong to $[\sigma,\rho]$ then there exists a unique $i$ such that $\imf{\phi}{t+s}\in\tau_{\eta_i,\sigma}$ and so\begin{linenomath}\begin{esn}
\imf{d}{\imf{\phi}{t+s},\imf{\phi}{t+s}\wedge\imf{\phi}{s}}
=\sum_{i} \indi{S_{x_i-}\leq t<S_{x_i}}\imf{f_i}{t-S_{x_i-}}.
\end{esn}\end{linenomath}We therefore obtain the representation
\begin{linenomath}\begin{align}
\label{deterministicItoRepresentation}
\imf{f}{s+t}-\imf{f}{s}
&=\imf{d}{\imf{\phi}{t+s},\rho}-\imf{d}{\imf{\phi}{s},\rho}
\\&=\imf{d}{\imf{\phi}{t+s},\imf{\phi}{s}\wedge \imf{\phi}{t+s}}-\imf{d}{\imf{\phi}{s},\imf{\phi}{s}\wedge \imf{\phi}{t+s}}\nonumber
\\&=\sum_{i}\indi{S_{x_i-}\leq t<S_{x_i}}\imf{f_i}{t-S_{x_i-}}-\Lambda_t.\nonumber
\end{align}\end{linenomath}

\begin{proof}[Proof of Theorem \ref{MeasuresWithSplittingPropertyCompactCaseTheorem}]
Let us start with the proof that the tree coded by the excursion measure of a L\'evy process has the splitting property. 
Since we have just identified right subtrees along $[\rho,\sigma)$ with excursions of the contour above its cumulative minimum, it suffices to prove that, for any $t>0$ and on the set $\set{\zeta>t}$, the excursions of $X_{\cdot+t}$ above its cumulative minimum form a Poisson point process. 
Let us recall that, under $\nu$, on the set $\set{X_t\in (0,\infty)}$ and conditionally on $X_t=x$, the path $X_{\cdot+t}$ has law $\q_x$. 
Hence, the point process with atoms at starting levels of the excursions and the excursions themselves is a
Poisson point process on $[0,x]\times \mathbf{E}$ with intensity $\leb\otimes \nu$.  
We conclude the splitting property of $\eta^\Psi$. 

Consider now a non-zero measure $\kappa$ on compact TOM trees with the splitting property and let $\nu$ be the push-forward of $\kappa$ by the contour map. 

We start by proving some consequences of the splitting property of $\kappa$. 
First, we prove that, in the compact case, the splitting property can be recast as
\begin{linenomath}
\begin{esn}
\imf{\kappa}{\imf{h}{X_t}e^{-\Xi_t^\infty g}}=\imf{\kappa}{\imf{h}{X_t}e^{-\int_{[0,X_t]\times \cts_c} \imf{g}{s,\tr{c}}\,\imf{\kappa}{d\tr{c}}\,ds}}. 
\end{esn}\end{linenomath}Indeed, note that if $\tr{c}$ is a compact TOM tree, 
then its height $\sup_{\sigma\in\tau}\imf{d}{\rho,\sigma}$ is finite. 
Hence, for any $r$ greater than the height, 
$\tr{c}$ coincides with its truncation $\tr{c}^r$ at level $r$. 
Therefore, for any measurable subset $A$ of the set of compact TOM trees, we have that
\begin{linenomath}
\begin{esn}
A=\lim_{r\to\infty} \set{\tr{c}: \tr{c^r}\in A}. 
\end{esn}\end{linenomath}We deduce that $\kappa^r$ converges to $\kappa$ in the following sense:\begin{linenomath}\begin{esn}
\imf{\kappa}{A}=\lim_{r\to\infty}\imf{\kappa^r}{A}. 
\end{esn}\end{linenomath}Hence, for any $g$ vanishing on trees with small height: \begin{linenomath}
\begin{align*}
\imf{\kappa}{\imf{h}{X_t}e^{-\Xi_t^\infty g}}
&=\lim_{r\to\infty}\imf{\kappa^r}{\imf{h}{X_t}e^{-\Xi_t^r g}}
\\&=\lim_{r\to\infty}\imf{\kappa^r}{\imf{h}{X_t}e^{-\int_{[0,X_t]\times \cts_c} (1-e^{-g(s,\tr{c})})\,\imf{\kappa^{r-s}}{d\tr{c}}\, ds}}
\\&=\imf{\kappa}{\imf{h}{X_t}e^{-\int_{[0,X_t]\times \cts_c} (1-e^{-g(s,\tr{c})})\,\imf{\kappa}{d\tr{c}}\, ds}}.
\end{align*}\end{linenomath}The identity between the extremes can then be generalized to non-negative and measurable $g$. 
We now prove the integrability conditions\begin{linenomath}\begin{esn}
\imf{\kappa}{\zeta>h}<\infty\text{ for }h>0
\quad\text{and}\quad
\int_0^\infty 1\wedge \imf{\zeta}{\tr{c}}\, \imf{\kappa}{d\tr{c}}<\infty.
\end{esn}\end{linenomath}Since $\kappa$ is non-zero, there exists $t>0$ such that $\imf{\kappa}{\zeta>t}>0$. 
Note that in $\Xi_t^\infty$, there can only be a finite number of atoms $(s_i,\tr{c}_i)$ such that $\tr{c}_i$ has measure $>\eps$ for any $\eps>0$. 
By the splitting property, applied with $h=\indi{[0,\infty)}$ and $\imf{g}{s,\tr{c}}=\indi{\imf{\zeta}{\tr{c}}>\eps}$, we get: 
\begin{linenomath}
\begin{align*}
0<
&\imf{\kappa}{\imf{h}{X_t}e^{-\# \text{ subtrees of measure $>\eps$ to the right of }[\rho, \imf{\phi}{t}]}}
\\&=\imf{\kappa}{\imf{h}{X_t}e^{-\int_{[0,X_t]\times \cts_c}(1-e^{-g(s,\tr{c})})\,\imf{\kappa}{d\tr{c}}\, ds}}
\\&=\imf{\kappa}{\imf{h}{X_t}e^{-(1-e^{-1})X_t\cdot \imf{\kappa\,}{\zeta>\eps}}}.
\end{align*}
\end{linenomath}We conclude that $\imf{\kappa}{\zeta>\eps}<\infty$ for any $\eps>0$. 
We now choose $g(s,\tr{c})= \imf{\zeta}{c}$. 
Next, under $\kappa$, the right of $\imf{\phi}{t}$ has finite measure almost everywhere, so that, using the fact that $\kappa$ has constant sojourn, say equal to $a$: 
\begin{linenomath}
\begin{align*}
0<\imf{\kappa}{\zeta>t}
&=\imf{\kappa}{\indi{\zeta>t, \text{ measure of the right of }\imf{\phi}{t}<\infty}}
\\&=\lim_{\lambda\to 0+}
\imf{\kappa}{\indi{\zeta>t} e^{-\lambda\text{ measure of the right of }\imf{\phi}{t}}}
\\&=\lim_{\lambda\to 0+}\imf{\kappa}{\imf{h}{X_t} e^{- a\lambda X_t-\lambda \Xi^\infty_t g}}
\\&=\lim_{\lambda\to 0+}\imf{\kappa}{\imf{h}{X_t} e^{- aX_t [\lambda+t \int (1-e^{-\imf{\zeta}{\tr{c}}})\,\imf{\kappa}{d\tr{c}}]  } }
\end{align*}
\end{linenomath}We conclude that 
$\int  1\wedge \imf{\zeta}{\tr{c}}\, \imf{\kappa}{d\tr{c}}<\infty$. 

Let $\Xi=\sum_i\delta_{\paren{x_i,f_i}}$ be a Poisson random measure on $[0,\infty)\times \mathbf{E}$ with intensity $\leb \otimes \nu$.  
Define the process $Y$ by the following procedure 
inspired by It\^o's synthesis theorem (cf. \cite{MR0402949}):
Let\begin{linenomath}\begin{esn}
S_l=al+\sum_{x_i\leq l} \imf{\zeta}{f_i}\quad\text{and}\quad \Lambda=S^{-1}.
\end{esn}\end{linenomath}We then let\begin{linenomath}\begin{esn}
Y_t=
\sum_{x_i\leq x}\indi{S_{x_i-}\leq t<S_{x_i}}\imf{f_i}{t-S_{x_i-}}-\Lambda_t
\end{esn}\end{linenomath}in analogy with \eqref{deterministicItoRepresentation}. 
Heuristically, $Y$ is the contour process of a tree obtained by grafting trees with law $\nu$ to the right of the vertical interval $(-\infty,0]$. We claim that $Y$ is a (sub)critical spectrally positive L\'evy process. 
If so, let $\Psi$ be its Laplace exponent. 
Since the compact TOM tree coded by the excursion measure of $Y$ above its cumulative minimum process has law $\kappa$ by construction, and if $Y$ is a (sub)critical L\'evy process this law should equal, by definition, $\eta^\Psi$, we see that $\kappa=\eta^\Psi$.


Let us prove that $Y$ is a (sub)critical spectrally positive L\'evy process. 
By construction, $Y$ has no positive jumps. 
Also, note that the running infimum of $Y$ is $-\Lambda$, which goes to $-\infty$. 
Hence, if $Y$ is proved to be a L\'evy process, then it is spectrally positive and is critical or subcritical. 
It remains to see that $Y$ has independent and stationary increments. 

Let us note that the process $Y^x$ obtained by killing $Y$ upon reaching $-x$ has the following construction in terms of Poisson random measures:
\begin{linenomath}\begin{esn}
Y^x_t=\begin{cases}
\sum_{x_i\leq x}\indi{S_{x_i-}\leq t<S_{x_i}}\imf{f_i}{t-S_{x_i-}}-\Lambda_t& t<S_{x}\\
\dagger& S_{x}\leq t
\end{cases}.
\end{esn}\end{linenomath}
Let $\tilde \q_x$ be the law of $x+Y^x$. 
By translation invariance of Lebesgue measure and independence properties of Poisson random measures we see that the two processes\begin{linenomath}\begin{esn}
\begin{cases}
x+Y^{x+y}_t&t<S_x\\
\dagger&t\geq S_x
\end{cases}
\quad\text{and}\quad 
\begin{cases}
x+y+Y^{x+y}_{t+S_x}& t<S_{x+y}-S_x\\
\dagger&t\geq S_{x+y}-S_x
\end{cases}
\end{esn}\end{linenomath}are independent and have respective laws $\tilde \q_x$ and $\tilde \q_y$. 
Also, note that if we concatenate a process with the same law as $Y$ to an independent process with law $\tilde \q_x$ when it gets killed, then we get a process with law $Y$. 

For any $t\geq 0$, let\begin{linenomath}\begin{align*}
\underline Y_t&=\inf_{s\leq t} Y_s=-\Lambda_t,\\
\quad g_t&=\sup\set{s\leq t: Y_s=\underline Y_s}=S_{\Lambda_t-}\\
\intertext{and} d_t&=\inf\set{s\geq t: Y_s=\underline Y_s}=S_{\Lambda_t}. 
\end{align*}\end{linenomath}Notice that the interval $[g_t,d_t]$ can be reduced to a point but that when it isn't, it is an interval of constancy for $\Lambda$. 
Also, note that  $g_t=d_t$ if and only if $S$ is continuous at $\Lambda_t$. 
Let $\Xi_t$  be the restriction of $\Xi$ to $[0,\Lambda_{t}]\times \mathbf{E}$ and define the random measure $\Xi^t$ characterized by
\begin{linenomath}\begin{esn}
\imf{\Xi^t}{A\times B}=\imf{\Xi}{\paren{A+\Lambda_{t}}\times B},
\end{esn}\end{linenomath}
where $A$ and $B$ are Borel subsets of $[0,\infty)$ and of the excursion space $\mathbf{E}$ respectively. 
Let $\G_l$ and $\G^l$ be the \sa s generated by the restriction of $\Xi$ to $[0,l]\times \mathbf{E}$ and $[l,\infty)\times \mathbf{E}$. 
Then $\G_l$ is independent of $\G^l$. 
It is simple to see that $\Lambda_t$ is a stopping time for $\paren{\G_l}$ for any $t\geq 0$ and using the independence of $\G_l$ and $\G^l$, 
the translation invariance of Lebesgue measure and discretizing $\Lambda_t$ by $\floor{\Lambda_t2^n}/2^n$, 
one shows the independence of $\Xi_t$ and $\Xi^t$ and that $\Xi^t$ has the same law as $\Xi$ by mimicking the proof of the strong Markov property for Feller processes. 
Define\begin{linenomath}\begin{esn}
\imf{\nu_l}{A}=\frac{\imf{\nu}{A, \zeta>l}}{\imf{\nu}{\zeta>l}}.
\end{esn}\end{linenomath}Using the analysis of the excursions straddling a given time of \cite[Ch. XIII\S 3, p. 488]{MR1725357}, 
we see that defining the \sa\ $\F^Y_{g_t}$ as $\sag{Y_{s\wedge g_t}: s\geq 0}$, 
then the law of the excursion of $Y$ straddling $t$ given $\F^Y_{g_t}$ equals\begin{linenomath}\begin{esn}
\indi{s<t}\nu_{t-s}+\indi{s=t}\delta_{\dagger}
 \end{esn}\end{linenomath}on the set $g_t=s$. 
The above discussion gives us access to 3 different parts of the path of $Y$: 
before $g_t$, 
between $g_t$ 
and $d_t$ (conditionally on $\F^Y_{g_t}$) and after $d_t$ (which, if we shift it to start it at zero, is independent of the pre-$d_t$ part and has the same law as $Y$ since it is constructed from $\Xi^t$ in the same manner as $Y$ is constructed from $\Xi$). 
We now study the measure $\nu$ to make a further description of the parts of the path between $g_t$ and $t$  and between $t$ and $d_t$.

Consider the shift operators $\theta_t$ on canonical space which take $f$ to the function $s\mapsto \imf{f}{t+s}$. 
Thanks to the splitting property of $\kappa$, we see that conditionally on $X_s$ and under $\nu$, the post-$s$ process has law $\tilde \q_{X_s}$. 
Formally 
the splitting property translates into\begin{linenomath}\begin{equation}
\label{splittingPropertyInductionBase}
\imf{\nu}{\imf{h}{X_s}F\circ\theta_s}=\imf{\nu}{\imf{h}{X_s}\imf{\tilde\q_{X_s}}{F}}
\end{equation}\end{linenomath}for nonnegative measurable $h$ 
and $F$ defined on $[0,\infty)$ and $\mathbf{E}$ respectively (with $g$ vanishing at $\dagger$). 
We now prove an analogous result under $\tilde\q_x$: 
\begin{linenomath}
\begin{equation}
\label{oneDimKilledMarkov}
\imf{\tilde \q_x}{\imf{h}{X_s} F\circ \theta_s}
=\imf{\tilde \q_x}{\imf{h}{X_s}\imf{\tilde \q_{X_s}}{F}}.
\end{equation}\end{linenomath}It is obtained from \eqref{splittingPropertyInductionBase} as follows. 
Fix $s>0$. 
Using the Poissonian description of $\tilde \q_x$, 
we see that under this measure, 
the trajectory splits into 3 independent subpaths: 
1) the subpath $X_{\cdot \wedge g_s}$ before $g_s$, 2) the subpath 
$X_{g_s+\cdot \wedge d_s}-X_{g_s-}$ between $g_s$ and $d_s$,
and 3) the subpath $X_{d_s+\cdot}$ after $d_s$. 
We now work conditionally on 
$\underline X_s=y, X_s=z, g_s=s'$; note that $X_{s-}=z$ also. 
After $d_s$, the process has law $\tilde \q_{y}$. 
Between $g_s$ and $d_s$, 
the process has law $\nu_{s-s'}$. 
Using \eqref{splittingPropertyInductionBase}, we see that under $\nu_{s-s'}$, 
the post $s-g_s$ part of the trajectory has law $\tilde \q_{z-y}$. 
By concatenation, we see that under $\tilde \q_x$ and on the set $\underline X_s=y, X_s=z, g_s=s'$, the post-$s$ part of the trajectory has law $\tilde \q_{X_s}$, which implies \eqref{oneDimKilledMarkov}. 
We deduce that $\tilde \q_x$ is Markovian 
and that, therefore $\nu$ also is. 

Finally, we finish the proof that $Y$ is a L\'evy process. 
Let $t>0$ and note that the process $Y$ has 3 independent subpaths by our analysis of its Poissonian construction: 
before $g_t$, between $g_t$ and $d_t$ (shifted by $-\Lambda_t$ to end at zero) and after $d_t$ (shifted by $-\Lambda_t$ to start at zero). 
Also, the law of the (shifted) process between $g_t$ and $d_t$ is $\nu_{t-g_t}$ and by the Markov property (under $\nu$) 
we see that the conditional law of the process between  $t$ and $d_t$ (minus $\Lambda_t$, to start at zero) given the pre-$t$ process is $\tilde \q_{X_t-L_t}$. 
By concatenating the post $d_t$-part, which has law the law of $Y$ (when shifted to start at zero)  
we see that the law of $Y_{t+\cdot}-Y_t$ 
given $\sag{Y_s:s\leq t}$ is the law of $Y$. 
Hence, $Y$ is a L\'evy process.
\end{proof}
\section{Locally compact TOM trees}
\label{LocallyCompactCRTSection}
We have already noted in Section \ref{IntroductionSection} that truncations of locally compact real trees are compact and defined the truncation of a locally compact TOM tree. 
In this short section, we explore the relationship between the contour processes of two truncations of a same locally compact TOM tree. 
We then use this information to see how to build locally compact TOM trees out of sequences of consistent compact TOM trees. 

The truncation at level $r$ of a locally compact TOM tree $\tr{c}$ is the compact TOM tree\begin{linenomath}\begin{esn}
\tr{c}^r=\paren{\paren{\tau^r,d|_{\tau^r},\rho},\leq|_{\tau_r},\mu|_{\tau^r}}.
\end{esn}\end{linenomath}We now study the relationship of the contour processes of two different truncations of the same tree.

\begin{proposition}
\label{HeightProcessOfTruncatedChronologicalTreeProposition}
Let $\tr{c}$ be a compact TOM tree. The contour processes $f$ and $f^r$ of $\tr{c}$ and $\tr{c}^r$ are related by a time-change as follows: 
let $\imf{C^r}{f}$ be the right-continuous inverse of the continuous and non-decreasing function $A^r=\imf{A^r}{f}$ given by\begin{linenomath}\begin{esn}
A^r_t=\int_0^t \indi{\imf{f}{s}\leq r}\, ds.
\end{esn}\end{linenomath}Then\begin{linenomath}\begin{esn}
f^r=f\circ C^r. 
\end{esn}\end{linenomath}
\end{proposition}
\begin{proof}
As verified at the end of the proof of Proposition \ref{HeightProcessCodesTreeProposition}, we have the equality $\mu=\leb\circ \phi^{-1}$, where $\phi$ is the exploration process of $\tr{c}$.
Hence\begin{linenomath}\begin{esn}
\imf{\mu^r}{L_\sigma}=\imf{\mu}{L_\sigma\cap\tau^r}=\imf{\leb}{[0,\imf{\mu}{L_\sigma}]\cap\set{t:\imf{d}{\imf{\phi}{t},\rho}\leq r}}=\imf{\leb}{\set{t\leq\imf{\mu}{L_\sigma}:\imf{f}{t}\leq r}}.
\end{esn}\end{linenomath}Let $\sigma\in\tau^r$ and define $s=\imf{\mu^r}{L_\sigma}$ and $t=\imf{\mu}{L_\sigma}$ so that\begin{linenomath}\begin{esn}
\imf{f^r}{s}=\imf{d}{\sigma,\rho}=\imf{f}{t}.
\end{esn}\end{linenomath}As we just computed,\begin{linenomath}\begin{esn}
s=\imf{\mu^r}{L_\sigma}=A^r_t.
\end{esn}\end{linenomath}However, by taking $\sigma_n\in [\rho,\sigma)$ satisfying $\sigma\leq\sigma_{n+1}\leq\sigma_n$ and setting $t_n=\imf{\mu}{L_{\sigma_n}}$, 
we see that $t_n$ decreases to $t$ and $\imf{f}{t_n}<r$, so that $A^r$ increases in any right neighborhood of $t$ and so\begin{linenomath}\begin{esn}
t=C^r_s.
\end{esn}\end{linenomath}Hence $f^r=f\circ C^r$ on the set $\set{\imf{\mu^r}{L_\sigma}:\sigma\in\tau^r}$, which by Lemma \ref{DensityOfImageLemma} is dense. We conclude that $f^r=f\circ C^r$ on $[0,\imf{\mu}{\tau^r}]$.
\end{proof}

We now use the preceding proposition to identify the space of locally compact TOM trees with the direct limit of compact TOM trees. 
Indeed, note that if $\tr{c}$ is a locally compact TOM tree and $\tr{c}^n$ is its truncation at height $n$, then the truncation of $\tr{c}^{n+1}$ at height $n$ equals $\tr{c}^n$.
Conversely, let $\paren{r_n,n\in\na}$ be a sequence of non-negative real numbers increasing to $\infty$ and $\paren{\tr{c}_n,n\in\na}$ be a sequence of compact TOM trees. 
\begin{definition}
The sequence  $\paren{\tr{c}_n,n\in\na}$ is said to be \defin{consistent} (at levels $\paren{r_n,n\in\na}$) if the truncation of $\tr{c}_{{n+1}}$ at height $r_n$ is isomorphic to $\tr{c}_n$. 
\end{definition}

\begin{proposition}
\label{inductiveLimitProposition}
If $\paren{\tr{c}_n,n\in\na}$ is a consistent sequence of compact TOM trees at levels $r_1<r_2<\cdots$ then there exists a unique locally compact TOM tree $\tr{c}$ such that the truncation of $\tr{c}$ at level $r_n$ is isomorphic to $\tr{c}_n$. 
\end{proposition}
\begin{proof}
Let us construct a locally compact TOM tree from a consistent sequence $\paren{\tr{c}_n,n\in\na}$ at levels $r_1<r_2<\cdots$; 
suppose that $\tr{c}_n=\paren{\paren{\tau_n,d_n,\rho_n},\leq_n,\mu_n}$, let $\psi_n$ be an isomorphism between $\tr{c}_n$ and the truncation of $\tr{c}_{n+1}$ at level $r_n$  (which is actually a mapping between $\tau_{n}$ and $\tau_{n+1}$ preserving distances, root, total order and which maps the measure $\mu_n$ into the restriction of $\mu_{n+1}$ to $\tau_{n+1}^{r_n}$). We then define $\psi_m^n$ for $m\leq n$ as the composition $\psi_{n-1}\circ\cdots\circ \psi_m$.


We begin by specifying the tree part $\tau$ of $\tr{c}$. 
Let $\tau$ be the direct limit of $\paren{\tau_n,n\in\na}$ with respect to the mappings $\paren{\psi_n}$. 
In other words, 
$\tau$ has elements of the type $\paren{l,\sigma}$ where $\sigma\in\tau_l$ and we identify $\paren{l,\sigma}$ and $\paren{m,\tilde \sigma}$ if there exists $n\geq l,m$ such that $\imf{\psi_l^n}{ \sigma}=\imf{\psi_m^n}{ \tilde \sigma}$. 
We then deduce that the preceeding equality holds for any $n\geq l,m$. 

Let us specify the distance $d$ to be placed on $\tau$: we define it by
\begin{linenomath}\begin{esn}
\imf{d}{\paren{l,\sigma},\paren{m,\tilde \sigma}}=\imf{d_n}{ \imf{\psi_l^{n}}{\sigma},\imf{\psi_m^{n}}{\tilde \sigma} }.
\end{esn}\end{linenomath}for any $n\geq l,m$. 

The root of $\tau$ is taken as the equivalence class of $(1,\rho_1)$. 

The triple $\paren{\tau,d,\rho}$ is locally compact rooted real tree. 
Indeed, note that if $\paren{l,\sigma}$ and $\paren{m,\tilde\sigma}$ are two representatives of elements of $\tau$, 
we can embed them in $\tau_{n}$ for any $n\geq m,l$ and since $\tau_n$ is a tree, 
we can obtain an isometry of an interval which starts at $\imf{\psi_l^n}{\sigma}$ and ends at $\imf{\psi_m^n}{\tilde \sigma}$. 
Lack of loops can be proved by noting that an isometry of an interval into $\tau$ must remain bounded, 
so that it can be transformed into an isometry of an interval into $\tau_n$ for some $n$. 
Local compactness and completeness can also be proved by a similar argument. 

A total order $\leq$ on $\tau$ can be defined by stating that\begin{linenomath}\begin{esn}
\paren{l,\sigma}\leq \paren{m,\tilde\sigma}\text{ if and only if } \imf{\psi_l^n}{\sigma}\leq_n \imf{\psi_m^n}{\tilde\sigma}
\end{esn}\end{linenomath}for some $n\geq l,m$. It is easy to see that $\leq$ satisfies \defin{Or1} and \defin{Or2}.


Finally, we place the measure $\mu$ on $\tau$. 
We abuse our notation to say that $\mu_n$ is the push-forward of the measure $\mu_n$ by the mapping that sends $\sigma$ to the equivalence class of $\paren{n,\sigma}$. 
Note that $\paren{\mu_n}$ increases and so we can define the measure $\mu$ by \begin{linenomath}\begin{esn}
\mu=\lim_{n\to\infty}\mu_n. 
\end{esn}\end{linenomath}It is simple to see that $\mu$ satisfies \defin{Mes1} and \defin{Mes2}.

We end the proof by noting that $\tr{c}=\paren{\paren{\tau,d,\rho},\leq,\mu}$ is a locally compact TOM tree whose truncation at height $r_n$ is isomorphic to $\tr{c}_n$. 
\end{proof}
Let $\paren{f_n}$ be a sequence of \cadlag\ functions where $f_n$ is defined on $[0,m_n]$, $\imf{f}{0}=0$ and such that $f_n$ codes a tree $\tr{c}_n$.
\begin{definition}
We say that the sequence $\paren{f_n}$ is \defin{consistent under time change} if there exists a sequence of heights $r_1,r_2,\ldots $ increasing to infinity and such that\begin{linenomath}\begin{esn}
f_{n}=f_{n+1}\circ \imf{C^{r_n}}{f_{n+1}}.
\end{esn}\end{linenomath}
\end{definition}

From Propositions \ref{HeightProcessOfTruncatedChronologicalTreeProposition} and \ref{inductiveLimitProposition} 
we see that (the equivalence class under isomorphism of) a locally compact TOM tree  can be identified with a sequence of functions consistent under time change (and such that their induced equivalence classes have zero Lebesgue measure). 
Let us then define the function $X^r_t$ defined on TOM trees and with values in $\re\cup\set{\dagger}$ 
by stating that $\imf{X^r_t}{\tr{c}}$ is the value at time $t$ of the contour process of the truncation $\tr{c}^r$ of $\tr{c}$ at level $r$ (if the measure of $\tr{c}^r$ is greater than $t$; otherwise we define it as $\dagger$). 
We can then use the \sa\ $\F=\sag{X^r_t,t,r\geq 0}$ when discussing measurability issues for functions defined on (or with values on) locally compact TOM trees. 
\begin{corollary}
\label{existenceOfProjectiveLimitOfMeasuresCorollary}
Let $\p^r$ be a sequence of measures on Skorohod space such that the push-forward of $\p^r$ under the time change at level $r'\leq r$ gives $\p^{r'}$. 
Suppose that under $\p^r$, 
the canonical process is non-negative,  approaches zero at death time and that
the equivalence classes $[s]_X$ induced by the canonical process have zero Lebesgue measure. 
Then there exists a unique measure $\p$ on locally compact TOM trees such that, under\ $\p$ the contour process of the truncation at level $r$ has law $\p^r$. 
\end{corollary}


\section{Consistent families of reflected L\'evy processes}
\label{ReflectedProcessSection}
The objective of this section is to give a further study of a family of stochastic processes consistent under time-change: the reflected L\'evy processes introduced in Section \ref{IntroductionSection}. 
This will allow us, through Corollary \ref{existenceOfProjectiveLimitOfMeasuresCorollary}, to construct random locally compact TOM trees. 
In this section, we give two additional constructions of reflected L\'evy processes.
Both constructions  make sense even in the case of positive killing coefficient. 
The first construction is based on time-changes and has a pathwise time-change consistency which is useful in constructing locally compact TOM trees. 
This enables us to construct the measure $\eta^\Psi$ in the statement of Theorem 
\ref{MeasuresWithSplittingPropertyCompactCaseTheorem} in the non-compact case. 
The second construction is through Poisson random measures. 
It is a consequence of the time-change construction and is fundamental to our proof of Theorem \ref{MeasuresWithSplittingPropertyCompactCaseTheorem} in the locally compact case. 

\subsection{Reflected L\'evy processes through time-change}
\label{timeChangeSubsection}
We place ourselves on the canonical space of \cadlag\ trajectories from $[0,\infty)$ to $\re$ together with the canonical process $X=\paren{X_s,s\geq 0}$ and the canonical filtration $\paren{\F_t,t\geq 0}$. 
Let $\p_x$ be the law of a spectrally positive L\'evy process (spLp) with Laplace exponent $\Psi$ started at $x$. 
We will suppose throughout that $X$ is not a subordinator under $\p_0$.

Let us consider a reflecting threshold $r$ and define the cumulative maximum process from $r$, $\overline X^r$ given by\begin{linenomath}\begin{esn}
\overline X_t^r= \paren{\max_{s\leq t}X_s-r}^+.
\end{esn}\end{linenomath}We will write $\overline X$ when $r$ is clear from the context, as in the next definition.
Recall that a spectrally positive L\'evy process with Laplace exponent $\psi$ started at $x\leq r$ and reflected below $r$ is a stochastic process having the law $\p_x^r$ of $X-\overline X$ under $\p_x$. 
Under $\p_x^r$, the canonical process is a strong Markov process as proved in 
Proposition 1 of \cite[Ch. VI]{MR1406564}. 
We will also make use of the L\'evy process reflected below $r$ and killed upon reaching zero and denote its law by $\q_x^r$; recall that killing preserves the strong Markov property. 
The laws $\q_x^r$ were introduced in Definition 4.5.4.1 of \cite{MR2466449} since they coincide with the law of the (jump-chronological) contour process $J^r$ of a splitting tree truncated at height $r$ associated to a finite-variation spLp with Laplace exponent $\Psi$ whose progenitor has lifetime $x$. 
The relationship between the contour processes at different heights (cf. Proposition \ref{HeightProcessOfTruncatedChronologicalTreeProposition}) suggests the following corresponding property of the laws $\q_x^r$: let\begin{linenomath}\begin{esn}
A^r_t=\int_0^t \indi{X_s\leq r}\, ds. 
\end{esn}\end{linenomath}Let $C^r$ be the right-continuous inverse of $A^r$.  
\begin{proposition}
\label{ConsistencyUnderTimeChangeProposition}
If $x\leq r_1<r_2$ then the law of $X\circ C^{r_1}$ under $\p_x^{r_2}$ is $\p_x^{r_1}$.
\end{proposition}
By stopping at the first hitting time of zero, we obtain:
\begin{corollary}
\label{ConsistencyUnderTimeChangeCorollary}
If $0\leq x\leq r_1<r_2$ then the law of $X\circ C^{r_1}$ under $\q_x^{r_2}$ is $\q_x^{r_1}$.
\end{corollary}
Indeed, \cite{MR2599603} proves Corollary \ref{ConsistencyUnderTimeChangeCorollary} when the underlying L\'evy process is of finite variation. 
We are interested in extending the consistency under time-change in order to construct random locally compact TOM trees out of the trajectories of more general L\'evy processes.

In the Brownian case (when $\imf{\Psi}{\lambda}=\lambda^2/2$), 
the equality in law between $X\circ C^{0}$ under $\p_x^{r_2}$ and $\p_x^{r_1}$ is reduced to the equality in law between Brownian motion reflected below its supremum 
and Brownian motion time changed by $C^0$ to remain negative. 
This can be proved using Tanaka's formula, Knight's Theorem and Skorokhod's reflection lemma. 
These tools are not available for general spectrally positive L\'evy processes. 
However, this already tells us that some  local time argument should be involved in the proof. 
Indeed, our argument is based on excursion theory for Markov processes, in which local time plays a fundamental role. 

Proposition \ref{ConsistencyUnderTimeChangeProposition} is a generalization of the results for Brownian motion or finite variation L\'evy processes. 
The only caveat is to see that the equality between time-change and reflection holds as a consequence of the assumed lack of negative jumps. 
For example, if $\p_0$ is the law of the difference of two independent stable subordinators of index $\alpha\in (0,1)$, then under $\p_x^{r_1}$, the barrier $r_1$ is touched. 
However, since points are polar  (see for example \cite[p. 63]{MR1406564}) then $X\circ C^{r_1}$ does not touch the barrier $r_1$ almost surely under $\p_x^{r_2}$ so that no equality in law can be valid in this case.

\begin{proof}[Proof of Proposition \ref{ConsistencyUnderTimeChangeProposition}]
Recall that the time-changed processes $X\circ C^{r_1}$ under $\p_x^{r_2}$ are also strongly Markovian, as shown for example in \cite{MR0193671} or \cite{MR958914}. 

The argument is based on excursion theory, as introduced in \cite[Ch. 4]{MR1406564}. 
Indeed, we prove that the (regenerative) sets of times at which the two (strong Markov) processes visit $r_1$, being the images of subordinators as in \cite[Ch. 2]{MR1746300}, have the same drift (zero in this case), and that the excursions of both processes below $r_1$ admit the same stochastic descriptions. 
To see that the drift is equal to zero, by Proposition 1.9 of \cite{MR1746300}, it suffices to see that the sets of times at which the two processes visit $r_1$ have zero Lebesgue measure almost surely. 

The law of $X$ under $\p_x^{r_1}$ until hitting $r_1$ is equal to the law of $X$ under $\p_x$ until it surpasses $r_1$. After $X$ hits $r_1$ under $\p_x^{r_1}$, the process has law $\p_{r_1}^{r_1}$, which is the push-forward of $\p_0^0$ under $f\mapsto r_1+f$. 
Similarly, $X$ equals $X\circ C^{r_1}$ until $X$ surpasses $r_1$, and under $\p_x^{r_2}$, these processes have the law of $X$ under $\p_x$ until it surpasses $r_1$. After $X\circ C^{r_1}$ hits $r_1$, it has the law under $\p_{x}^{r_2}$ of $X\circ C^{r_1}$ under $\p_{r_1}^{r_2}$, which by spatial homogeneity is the image of the law of $X\circ C^0$ under $\p_{0}^{r_2}$. 
Hence, it suffices to focus on the case $0=x=r_1<r_2$; set $r_2=r$ to simplify notation. 


Formula 9.2.9 of \cite[p. 98]{MR2320889} implies that the Laplace exponent of the upward ladder time process of $X$ under $\p_0$ is equal to $\alpha/\imf{\Phi}{\alpha}$ where $\Phi$ is the right-continuous inverse of $\Psi$ 
and since $\p_0$ is not the law of a subordinator, then $\imf{\Phi}{\alpha}\to\infty$ as $\alpha\to\infty$ implying that the drift of the upward ladder time process, 
which is equal to $\lim_{\alpha\to\infty} 1/\imf{\Phi}{\alpha}$, 
is zero. 
Hence, the set of times $X$ is at its supremum under $\p_0$ has zero Lebesgue measure almost surely and so $\int_0^\infty \indi{X_s=0}\, ds=0$ almost surely under $\p_0^0$.

Let us obtain a corresponding statement for $X\circ C^0$ under $\p_0^r$. 
Since the time-change just removes parts of the trajectory, 
we are reduced to proving that $X$ spends zero time at $0$ under $\p_0^r$ which is equivalent 
to proving that $X-\overline X^r$ spends zero time at zero under $\p_0$. 
Note however that $X$ spends zero time at $0$ under $\p_0$, even if $X$ is compound Poisson minus a drift since the drift is forced to be nonzero having excluded subordinators. 
Hence, until the trajectory of $X-\overline X^r$ reaches $r$, it spends zero time at zero under $\p_0$. 
After this process reaches $r$, its law is that of $r+X-\overline X$ under $\p_0$ .
Note that
\begin{linenomath}
\begin{esn} 
0=
\imf{\se_{\p_0}}
{\int_0^\infty \indi{X_t-\overline X_t=-r}\, dt}
\text{ if and only if   }
0=\imf{\se_{\p_0}}{\int_0^\infty\lambda e^{-\lambda t} \indi{X_t-\overline X_t=-r}\, dt}. 
\end{esn}\end{linenomath}Hence, we now show that if $T$ is a standard exponential time independent of $X$ then $X_T-\overline  X_T$ is almost surely not $-r$. 
By the Pe{\v{c}}erski{\u\i}-Rogozin formula (cf. Th. 5 in \cite[Ch. VI]{MR1406564} or Equation (5) in \cite{MR2978134}), we know that $X_T-\overline X_T$ is a negative infinitely divisible random variable with zero drift and L\'evy measure equal to\begin{linenomath}\begin{esn}
\imf{\nu}{A}=\int_0^\infty \frac{e^{-t}}{t}\imf{\p_0}{X_t\in A\cap (-\infty,0)}\, dt. 
\end{esn}\end{linenomath}Recall that since $X$ is not a subordinator under $\p_0$ then  $0$ is regular for $(-\infty,0)$, as verified in Theorem 1 of \cite[Ch. VII]{MR1406564}. 
By Rogozin's criterion for regularity (cf. Proposition 11 of \cite[Ch. VI]{MR1406564} or equation (8) in \cite{MR2978134}), we then see that\begin{linenomath}\begin{esn}
\imf{\nu}{(-\infty,0)}=\int_0^\infty \frac{e^{-t}}{t}\imf{\p_0}{X_t<0}\, dt=\infty. 
\end{esn}\end{linenomath}Since $\nu$ is an infinite measure  the law of $X_T-\overline X_T$ has no atoms under $\p_0$ (cf. Theorem 27.4 in \cite{MR1739520}). 


Since the laws $\p_0^r$ and $\p_0^0$ are strongly Markovian, the sets\begin{linenomath}\begin{esn}
\set{t\geq 0:X\circ C^0_t=0}\quad\text{and}\quad \set{t\geq 0:X_t=0}
\end{esn}\end{linenomath}are regenerative under $\p_0^r$ and $\p_0^0$; 
we have just proved that in both cases their corresponding inverse local time at zero 
has zero drift. 

We now describe the excursions below $0$. 
Let $\mathbf{E^-}$ be the set of negative excursions. 
That is, the set of functions $\fun{f}{[0,\infty)}{(-\infty,0]\cup\set{\dagger}}$ for which there exists $\zeta=\imf{\zeta}{f}>0$, termed the lifetime, such that $f$ is strictly negative  and \cadlag\ with no negative jumps on $(0,\zeta)$ and equal to $\dagger$ after $\zeta$; hence, $\dagger$ is again interpreted as a cemetery state. The set
$\mathbf{E^-}$ is equipped with the natural filtration of the canonical process. 
Since we have proved that the canonical process $X$ under $\p_x^{r_1}$ 
(resp. $X\circ C^{r_1}$) under $\p_x^{r_2}$) spends zero time at $r_1$, excursion theory tells us that $X$ under $\p_x^{r_1}$ has the same law as the stochastic process $Z$ given by 
\begin{linenomath}
\begin{esn}
Z_t=\sum_{i}\indi{\tau_{\Lambda_i-}\leq t<\tau_{\Lambda_i}} \imf{E_i}{t-\tau_{\Lambda_i-}},
\end{esn}\end{linenomath}where the random measure $\sum_{i}\delta_{(\Lambda_i,E_i)}$ is a Poisson random measure on $[0,\infty)\times \mathbf{E^-}$ with intensity $\leb\otimes \nu$ for some $\sigma$-finite measure $\nu=\nu^0$ (resp. $\nu=\nu^r$) on $\mathbf{E^-}$ and
\begin{linenomath}
\begin{esn}
\tau_l=\sum_{\Lambda_i\leq l}\imf{\zeta}{E_i}. 
\end{esn}\end{linenomath}The measure $\nu^0$ (resp. $\nu^r$) is characterized, 
up to a multiplicative constant which does not affect the law of $Z$, 
as follows. 
Let $T_{\eps}$ denote the hitting time of $-\eps$ by the canonical process. 
Since functions in $\mathbf{E^-}$ start at zero 
and leave 0 immediately, 
any measure on $\mathbf{E^-}$ is determined by its values on 
$\cup_{\eps>0}\sag{X_{t+T_{\eps}}: t\geq 0}$. 
Let $\tilde \p_0^r$ be the law of $X\circ C^0$ under $\p_0^r$. 
Let\begin{linenomath}\begin{esn}
\mc{O}=\set{t\geq 0: X_t<0}.
\end{esn}\end{linenomath}Since $X$ has non-negative jumps under $\p_0^0$ 
and $\tilde \p_0^{r}$ 
then $\mc{O}$ is almost surely open under both measures. 
Its connected components are termed the excursion intervals of $X$ below $0$. 
If $(a,b)$ is a connected component of $\mc{O}$ then the non-negativity of jumps forces $X_{a}=0$, and so the path $e$ given by $e_s=X_{a+s}$ if $0\leq s<b-a$ and $e_s=\dagger$ if $s\geq b-a$ is an element of $\mathbf{E^-}$. 
The depth of the excursion $e$ is defined as $\min_{s<b-a} e_s$. 
For any $\eps>0$, there is a well-defined sequence of successive excursions of depths falling below $-\eps$ which are iid with common laws $\nu^0_\eps$ 
(respectively $\nu^r_\eps$) under $\p_0^0$ (respectively under $\tilde\p_0^r$); 
this follows from the strong Markov property under $\p_0^0$ and $\tilde \p_0^r$. 
Then, the intensity measure $\nu^0$ is the unique $\sigma$-finite measure
(up to a multiplicative constant that does not affect the law of $Z$) 
such that $\nu^0$ is finite on the set of excursions falling below $-\eps$ and $\nu^0_\eps$ equals $\nu^0$ conditioned on falling below $-\eps$. 
Indeed, the measure $\nu^0$ can be defined by\begin{linenomath}\begin{esn}
\nu^0=\lim_{\eps\to 0} \frac{\nu^\eps}{\imf{\nu^\eps}{\text{depth}< 1}}.
\end{esn}\end{linenomath}
We construct $\nu^r$ in the same manner, based on the probability measures $\nu_{\eps}^r$ giving us the law of successive excursions of $X\circ C^0$ below zero of depth below $-\eps$ under $\p_0^r$. 
Hence, to prove that $\nu^0=\nu^r$, it suffices to see that $\nu^0_\eps=\nu^r_\eps$.
Let us describe the left-hand side: 
under $\p_0$, consider the process $X\circ \theta_{T_\eps}$ until it becomes positive, 
after which we send it to $\infty$; this law equals the image $\p_{-\eps}$ by killing upon becoming positive by lack of positive jumps and this is exactly the measure $\nu^0_\eps\circ \theta_{T_\eps}^{-1}$. 
However, the law $\nu^r_\eps\circ\theta_{T_\eps}$ admits the same description. 
Indeed, it suffices to note that if we define the stopping time\begin{linenomath}\begin{esn}
T^0_\eps=\inf\set{t\geq 0:X\circ C^0_t=-\eps}
\end{esn}\end{linenomath}for the process $X\circ C^0$ and note that the process $X(C^0_{T^0_\eps+\cdot})$  is equal to $X_{T_\eps+t}$ until the latter becomes positive.  
Hence $\nu^0=\nu^r$ and so $\p_0^0$ equals the law of $X\circ C^0$ under $\p_0^r$. 
Note that the proof remains valid even in the subcritical case where there are excursions of infinite length. 
\end{proof}
As mentioned in \cite{MR2466449}, Corollary \ref{ConsistencyUnderTimeChangeCorollary} and Kolmogorov's consistency theorem 
imply the existence of  a doubly indexed process $X^{r}_t$ such that $X^r$ has law $\p_x^r$ for every $r\geq x$ and which is consistent under time-change in the following pathwise sense:\begin{linenomath}\begin{esn}
X^r=X^{r'}\circ C^{r',r}
\end{esn}\end{linenomath}for $r<r'$ where $C^{r',r}$ is the inverse of $A^{r',r}$ and\begin{linenomath}\begin{esn}
A^{r',r}_t=\int_0^t \indi{X^{r'}_s\leq r}\, ds.
\end{esn}\end{linenomath}We now give a representation of this doubly indexed process by time-changing independent copies of a L\'evy process. 

%
Let $X^1,X^2,\ldots$ be independent copies of $X$ and define: \begin{linenomath}\begin{esn}
T_1=\inf\set{t\geq 0:x+X^1_t=0},\quad T_n=\inf\set{t\geq 0:r+X^n_t=0}\text{ for $n\geq 2$}
\end{esn}\end{linenomath}as well as\begin{linenomath}\begin{esn}
A^{1}_t=\int_0^t\indi{x+X^1_{s\wedge T_1}\leq r}\, ds
\end{esn}\end{linenomath}and, for $n\geq 2$,\begin{linenomath}\begin{esn}
A^n_t=\int_0^t\indi{r+X^n_{s\wedge T_n}\leq r}\, ds,
\end{esn}\end{linenomath}Let $C^n$ be the right continuous inverse of $A^n$ for any $n\geq 1$. 
Consider the process\begin{equation}
\label{timeChangeConstructionOfRelfectedLevyProcessEquation}
Y^r_t=\begin{cases}
x+X^1\circ\alpha^1_t&\text{if $t<C^1_\infty$}\\
r+X^n\circ\alpha^n_{t-C^1_\infty-\cdots-C^{n-1}_\infty}&\text{if $C^1_\infty+\cdots+C^{n-1}_\infty\leq t<C^1_\infty+\cdots+C^{n}_\infty$ for $n\geq 2$}
\end{cases},
\end{equation}which, by construction, is consistent under time-change. 
\begin{proposition}
\label{timeChangeConstructionOfReflectedAndKilledLevyProcessProposition}
The process $Y^r$ has law $\q_x^r$.
\end{proposition}
\begin{proof}
Note that if $X$ is (sub)critical, 
then $C^1_\infty=\infty$ since $T<\infty$;
hence, we only need $X^1$ in the definition of $Y^r$. 
But then, $Y^r$ coincides with $X\circ C^r$ under $\q_x^r$, so that the result is valid in this case. 

It remains to consider the supercritical case. 
Also, before $Y^r$ hits $r$ it has the same law  as the L\'evy process until it exceeds $r$, so that it suffices to consider the case $x=r$.

If $\tilde Y^r$ admits the same construction as $Y^r$ but without killing upon reaching zero (that is, we use $\tilde A^n_t=\int_0^t\indi{r+X^n_{s}\leq r}\, ds$ instead of $A^n$),  we now prove that $\tilde Y^r$ has law $\p_r^r$, which proves the stated result.

The process $\tilde Y$ is strongly Markovian and, using the proof of Proposition \ref{ConsistencyUnderTimeChangeProposition}, one sees that 
\begin{linenomath}\begin{esn}
\int_0^t \indi{Y_s^r=r}\, ds=0
\end{esn}\end{linenomath}almost surely. 

Also, for each $\eps>0$, the successive excursions of $Y^r$ reaching $r-\eps$ are also the excursions of some $X^i$ (below $r$ and reaching $r-\eps$). 
After they reach $r-\eps$ their law is the image of  $\p_{r-\eps}$ by killing upon reaching $(r,\infty)$. 
Therefore, the intensity of excursions of $Y^r$ below $r$ is $\nu^0$. 
By the arguments of the proof of Proposition \ref{ConsistencyUnderTimeChangeProposition}, 
we see that $Y^r$ has law $\p_r^r$.
\end{proof}

\subsection{Poissonian construction of reflected L\'evy processes}
\label{poissonianConstructionSubsection}
We now present a Poissonian construction of reflected L\'evy processes which will be useful in 
establishing Theorem \ref{MeasuresWithSplittingPropertyCompactCaseTheorem} in the locally compact case. 

Let $\Psi$ be the Laplace exponent of a spectrally positive L\'evy process which is not a subordinator 
and let $\Phi$ be its right-continuous inverse.
Let $d$ be the drift of $\Phi$, so that\begin{linenomath}\begin{esn}
d=\lim_{\lambda\to\infty}\frac{\imf{\Phi}{\lambda}}{\lambda}
\end{esn}\end{linenomath} 

As before, we let $\p_x^r$ be the law of a $\Psi$-L\'evy process reflected under $r>0$ and started at $x\leq r$. 
Suppose for the moment that $\Psi$ is subcritical or critical (that is, $\imf{\Psi'}{0+}\geq 0$). 
Then we can perform the following Poissonian construction of $\p_x^r$: 
let $\nu$ be the excursion measure of $X-\underline X_t$ under $\p_0$ and define 
$\nu^r$ as the push-forward of $\nu$ by the truncation operator above $r$ (which sends $f$ to $f\circ \imf{C^r}{f}$). 

Let\begin{linenomath}\begin{esn}
\Xi^r=\sum_{n}\delta_{\paren{s_n,f_n}}
\end{esn}\end{linenomath}be a Poisson random measure on $[0,\infty)\times \mathbf{E}$ with intensity $\imf{\leb}{ds}\otimes \imf{\nu^{r-x+s}}{df}$ and define\begin{equation}
\label{ReconstructionReflectedLevyProcessFromExcursionsEquation}
S^r_t=d t+\sum_{s_n\leq t}\imf{\zeta}{f_n},\quad
M^r=-\paren{S^r}^{-1}
\quad\text{and}\quad
W_t=x+M^r_t+\sum_{n}\indi{S^r_{s_n-}\leq t<S^r_{s_n}}\imf{f_n}{t-S^r_{s_n-}}
\end{equation}
\begin{proposition}
\label{poissonianConstructionProposition}
For a (sub)critical Laplace exponent $\Psi$, 
the law of $W$ is $\p_x^r$.
\end{proposition}
\begin{proof}
It suffices to consider the case $x=0$ because the push-forward of $\p_0^{r-x}$ by the operator which adds $x$ to a given function is $\p_x^{r}$. 

Recall that under $\p_0$, $-\underline X$ is the local time of $X-\underline X$ at zero and its right-continuous inverse, say $S$, is a subordinator with Laplace exponent $\Phi$, as shown in \cite[Ch. VII]{MR1406564}. 
Also, the point process of excursions of $X-\underline X$ above zero  under $\p$, say\begin{linenomath}\begin{esn}
\Xi'=\sum_{n}\delta_{\paren{s_n',f_n'}},
\end{esn}\end{linenomath}is a Poisson point process on $[0,\infty)\times \mathbf{E}$ with intensity $\imf{\leb}{ds}\otimes \imf{\nu}{df}$. 
Finally, since excursion intervals become jumps of $S$ and $S$ is a subordinator with drift $d$, $X$ and $S$ can be recovered from $\Xi'$ and $d$ by means of the formulae
\begin{linenomath}\begin{align*}
S_t&=d t+\sum_{s'_n\leq t}\imf{\zeta}{f'_n}
, \quad
-\underline X
=S^{-1}
\intertext{and}
X_t-\underline X_t&=\sum_{n>0}\indi{S_{s'_n-}\leq t<S_{s'_n}}\imf{f'_n}{t-S_{s'_n-}}.
\end{align*}\end{linenomath}
Since the law of $X\circ C^r$ under $\p_0$ is $\p^r_0$, 
we only need to see that the excursion process of $X\circ C^r$ has the same law as $\Xi^r$ 
and that $X\circ C^r$ can be written in terms of its point process of excursions as in \eqref{ReconstructionReflectedLevyProcessFromExcursionsEquation}.

An excursion of $X\circ C^r$ above its cumulative minimum, when the latter equals $s$ is an excursion of $X-\underline X$ reflected below level $r+s$ by time-change. 
Hence the point process of excursions of $X\circ C^r$ above its cumulative minimum is\begin{linenomath}\begin{esn}
\tilde \Xi=\sum_n\delta_{\paren{\imf{\zeta}{f'_n\circ C^r},f'_n\circ C^r}},
\end{esn}\end{linenomath}which is a Poisson point process with intensity $\imf{\leb}{ds}\otimes \imf{\nu^{r+s}}{df}$. 

Also, the cumulative minimum process of $X\circ C^r$ is obtained from $\underline  X$ by erasing the time projection of excursions above the reflection threshold, which means erasing parts of jumps of $S$. 
Hence, the process\begin{linenomath}\begin{esn}
S^r=\paren{-\underline X\circ C^r}^{-1}
\end{esn}\end{linenomath}is given by
\begin{linenomath}\begin{esn}
S^r_t=dt+\sum_{s_n\leq t} \imf{\zeta}{f'_n\circ \imf{C^{r-s'_n}}{f'_n}}.
\end{esn}\end{linenomath}Finally, $X\circ C^r$ minus its cumulative minimum can be obtained by concatenating excursions,\begin{linenomath}\begin{esn}
X\circ C^r_t-\min_{s\leq t}X\circ C^r_s=\sum_{s'_n} \indi{S^r_{s'_n-}\leq t<S^r_{s'_n}}\imf{f'_n\circ C^{r-S^r_{s'_n-}}}{t-S^r_{s'_n-}}
\end{esn}\end{linenomath}so that indeed $X\circ C^r$ can be reconstructed from its excursion process as in \eqref{ReconstructionReflectedLevyProcessFromExcursionsEquation}.
\end{proof}

When $\Psi$ is supercritical, let $b>0$ be the largest root of the convex function $\Psi$ and $\imf{\Psi^\#}{\lambda}=\imf{\Psi}{\lambda+b}$. 
Recall that $\Psi^\#$ is the Laplace exponent of the (subcritical) spectrally positive L\'evy process with Laplace exponent $\Psi$ conditioned to reach arbitrarily low levels in the sense that its law equals\begin{linenomath}\begin{esn}
\lim_{z\to-\infty}\imf{\p_0}{\cond{\cdot}{\underline X_\infty<z}}
\end{esn}\end{linenomath}on every $\F_s$. 
(Cf. Lemma 7 in \cite[Ch. VII]{MR1406564} and Lemme 1 in \cite{MR1141246}.) 
Let us use the notation $\p_x$ and $\p_x^\#$ to distinguish both L\'evy processes. 
Let $H_a$ denote the hitting time of $-a$, defined by $H_a=\inf\set{t\geq 0: X_t\leq -a}$. 
Recall that conditionally on $H_a<\infty$, 
the laws of $X$ stopped at $H_a$ coincide under $\p$ and $\p^\#$. 
We deduce that, conditionally on $H_a<\infty$, the point process of excursions of $X$ above its cumulative minimum that start at levels deeper than $-a$ under $\p$ 
is a Poisson point process with the same intensity as under $\p^\#$ (let us call it $\nu^\#$). 
Furthermore, $-\underline X_\infty$ under $\p$ has an exponential law with rate $b$ 
and the post minimum process is independent of the pre-minimum process (cf. Theorem 25 in \cite[Ch. 8, p.85]{MR2320889}). 
The law of the former had been denoted $\p^{\rightarrow}$.
With these preliminaries, we can give a construction of the excursion process of $X$  above its cumulative minimum $\underline X$ under $\p$ and $\p^r$. Let\begin{linenomath}\begin{esn}
\Xi=\sum_n \delta_{\paren{s_n,f_n}}
\end{esn}\end{linenomath}be a Poisson point process with intensity $\nu^\#+b \p^\rightarrow$, interpreting the law $\p^\rightarrow$ as that of an excursion of infinite length. Construct the processes\begin{linenomath}\begin{esn}
S_t=dt+\sum_{s_n\leq t}\imf{\zeta}{f_n},\quad M=-\paren{S^{-1}}\quad\text{and}\quad W_t=x+M_t+\sum_{S_{s_n-}\leq t<S_{s_n}}\imf{f_n}{t-S_{s_n-}}. 
\end{esn}\end{linenomath}
\begin{proposition}
If $\Psi$ is supercritical, the law of $W$ is $\p_x$. 
\end{proposition}An analogous Poissonian construction is valid in the reflected case. 
To construct the excursion measure, let $V^r$ 
stand for the operator that removes (by time change) the trajectory above  level $r$ and 
let $\nu^r$ be the law of the concatenation of $\p^\rightarrow\circ (V^r)^{-1}$ followed by $\q_r^r$ (which is the law of the L\'evy process started at $r$, reflected below $r$, and killed (or sent to $\dagger$) upon reaching zero. 
Consider also the push-forward $\nu^{\#, r}$ of $\nu^{\#}$ by truncation at level $r$. Finally, let\begin{linenomath}\begin{esn}
\Xi^{r}=\sum_{n}\delta_{\paren{s_n,f_n}}
\end{esn}\end{linenomath}be a Poisson point process with intensity $\imf{\leb}{ds}\otimes\paren{\nu^{\#,r-x+s}+b \nu^{r-x+s}}$. As before, consider\begin{linenomath}\begin{esn}
S^r_t=dt+ \sum_{s_n\leq t} \imf{\zeta}{f_n}, \quad M^r=-\paren{S^r}^{-1}\quad\text{and}\quad W_t^r=x+M^r_t+\sum_{S^r_{s_n-}\leq t<S^r_{s_n}}\imf{f_n}{t-S^r_{s_n-}}.
\end{esn}\end{linenomath}
\begin{proposition}
\label{PoissonianConstructionRelfectedLevyProcessProposition}
If $\Psi$ is supercritical, 
the law of $W^r$ is $\p_x^r$. 
\end{proposition}
\begin{proof}[Proof sketch]
Again, it suffices to consider the case $x=0$. 
The arguments are very close to those of Proposition \ref{poissonianConstructionProposition}. 
The differences stem from the fact that in the supercritical case, under $\p_0$, 
there are two types of excursions: those corresponding to the associated subcritical L\'evy process with exponent $\Psi^{\#}$, and the unique excursion of infinite length (which appears at rate $b$ for the cumulative minimum process). 
Appealing to the time-change construction of $\p^r_0$
in the proof of Proposition \ref{timeChangeConstructionOfReflectedAndKilledLevyProcessProposition} we also see the two type of excursions. 
The first type of excursions come from the associated subcritical process, 
which are then affected by a time-change to remain below level $r-x+s$ if the cumulative minimum takes the value $-s$. 
On the other hand, the excursion of infinite length under $\p_0$ will also be affected by a time-change to remain below $r-x+s$, except that after death time, we must concatenate a process with law $\p_r^r$ which is killed upon reaching $-s$. 
These second types of excursions arrive at rate $b$. 
\end{proof}

\section{Measures on locally compact TOM trees with the splitting property}
\label{locallyCompactSplittingTreeSection}

In this section, we will establish the relationship between L\'evy processes and measures on locally compact TOM trees stated in Theorem \ref{MeasuresWithSplittingPropertyCompactCaseTheorem}. 
Part of the proof of this theorem is similar to the proof in the compact case of Theorem \ref{MeasuresWithSplittingPropertyCompactCaseTheorem} which we will follow closely, highlighting the differences.

We first use excursion theory for L\'evy processes as in the compact case to construct the measure $\eta^\Psi$ of Theorem \ref{MeasuresWithSplittingPropertyCompactCaseTheorem} in the non-compact case and to prove that it has the splitting property. 

\begin{proof}[Proof of Theorem \ref{MeasuresWithSplittingPropertyCompactCaseTheorem} in the non-compact case: construction of $\eta^\Psi$ and proof of its splitting property]
Let $\Psi$ be a supercritical Laplace exponent. 
Recall the definition of the truncated excursion measure $\nu^r$ associated to $\Psi$ given in Section \ref{IntroductionSection} before the definition of the splitting property. 


From Proposition \ref{levyProcessCodesTreeProposition}, the local absolute continuity of $\p^{\Psi}$ and $\p^{\#}$, and the fact that $Y^r$ is a concatenation of time-changed trajectories of $\p^\Psi_0$ 
 we see that 
the canonical process codes a compact TOM tree $\nu^r$-almost surely. 
Let $\eta^r$ be the image measure of $\nu^r$ under the mapping sending functions to TOM trees. 
Thanks to Corollary \ref{ConsistencyUnderTimeChangeCorollary}, we see that the measures $\paren{\nu^r,r\geq 0}$ are consistent under time-change. 
Hence, Corollary \ref{existenceOfProjectiveLimitOfMeasuresCorollary} implies the existence of a measure $\eta^\Psi$ on locally compact TOM trees 
such that  the image of $\eta^\Psi$ by truncation at level $r$ equals $\eta^r$. 
By the Markov property under $\nu^r$, we see that on the set $\zeta>t$ and conditionally on $X_t=x$, the post-$t$ process has law $\q_x^r$. 
By the Poissonian construction of $\p_x^r$ of Proposition \ref{PoissonianConstructionRelfectedLevyProcessProposition}, 
we see that the point process of excursions above the cumulative minimum of the post-$t$ process 
is a Poisson random measure with intensity $\imf{\leb}{ds}\otimes \imf{\nu^{r-x+s}}{df}$. 
We conclude that $\eta^\Psi$ has the splitting property. 
\end{proof}

We now consider a measure $\eta$ on locally compact TOM trees which is not concentrated on compact ones. 
Our aim is to construct the Laplace exponent of a supercritical spectrally positive L\'evy process such that $\eta=\eta^\Psi$. 
In the compact case, we could construct the corresponding L\'evy process directly, while in the locally compact case, 
we will construct first the L\'evy process reflected at a given level $r$ and obtain the L\'evy process in the limit as $r\to\infty$. 

\begin{proof}[Proof of Theorem \ref{MeasuresWithSplittingPropertyCompactCaseTheorem} in the non-compact case, Construction of a Laplace exponent such that $\eta=\eta^\Psi$]
Let $\eta^r$ be the push-forward of $\eta$ by truncation at level $r$ and let $n^r$ the law of the contour process under $\eta^r$. Let\begin{linenomath}\begin{esn}
\Xi=\sum_i \delta_{\paren{x_i,f_i}}
\end{esn}\end{linenomath}be a Poisson random measure on $[0,\infty)\times \mathbf{E}$ with intensity $\imf{\leb}{ds}\otimes \eta^{r+s}$ (the definition of the intensity was simpler in the compact case since truncation was not necessary). 
If $a$ is the sojourn of $\eta$, define\begin{linenomath}\begin{esn}
S_l=al+ \sum_{x_i\leq l}\imf{\zeta}{f_i}\quad\text{and} \quad \Lambda=S^{-1}
\end{esn}\end{linenomath}as well as\begin{linenomath}\begin{esn}
Y^r_t=\sum_{x_i\leq x}\indi{S_{x_i-}\leq t<S_{x_i}}\imf{f_i}{t-S_{x_i-}}-\Lambda_t. 
\end{esn}\end{linenomath}We claim that $Y^r$ is a reflected L\'evy process. 
To prove it, however, we need to let $r\to\infty$ to see that we get a limit which is a L\'evy process. 
However, truncation allows us to use the splitting property in a way that is parallel to the compact case. 
Note that the measure $\eta^r$ satisfies an $r$-dependent version of the splitting property: 
under $\eta^r$ and conditionally on $X_t=x$, $\Xi_t$ is a Poisson random measure on $[0,x]\times \cts_c$ with intensity $\imf{\leb}{ds}\otimes \imf{\eta^{r-s}}{df}$. 
This is an important difference with the compact case. 
For example, if we denote by $Y^{r,x}$ the process obtained by killing $Y^r$ when it reaches level $-x$, 
then (with the same proof as in the compact case) we see that the two processes\begin{linenomath}\begin{esn}
\begin{cases}
Y^{r,x+y}_t& t\leq S_x\\
\dagger&t\geq S_x
\end{cases}
\quad\text{and}\quad
\begin{cases}
Y^{r,x+y}_{t+S_x}+x& t\leq S_{x+y}-S_x\\
\dagger&t\geq S_{x+y}-S_x
\end{cases}
\end{esn}\end{linenomath}are independent and have the same laws as $Y^{r,x}$ and $Y^{r+x,y}$ respectively. 

Denote the law of $x+Y^{r-x,x}$ by $\q_x^r$. 
Just as in the compact case, use of the splitting property implies that both $n^r$ and $\q_x^r$ are Markovian: 
\begin{linenomath}\begin{align*}
\imf{n^r}{\imf{g_1}{X_{s_1}}\cdots\imf{g_n}{X_{s_n}}F\circ\theta_{s_n}}
&=\imf{n^r}{\imf{g_1}{X_{s_1}}\cdots\imf{g_n}{X_{s_n}}\imf{\q_{X_{s_n}}^r}{F}}
\intertext{and}
\imf{\q_x^r}{\imf{g_1}{X_{s_1}}\cdots\imf{g_n}{X_{s_n}}F\circ\theta_{s_n}}
&=\imf{\q_x^r}{\imf{g_1}{X_{s_1}}\cdots\imf{g_n}{X_{s_n}}\imf{\q_{X_{s_n}}^r}{F}}.
\end{align*}\end{linenomath}

Let $\p_x^r$ be the law of $x+Y^{r-x}$ (defined for $x\leq r$). 
The above Markov property for $\q_x^r$ implies that $\p_x^r$ is also Markovian. 
Using Proposition \ref{HeightProcessOfTruncatedChronologicalTreeProposition} we see that if $r<r'$ then the image of $n^{r'}$ under truncation at level $r$ equals $n^r$. 
Hence, the image of $\p_x^{r'}$ under truncation at level $r$ equals $\p_x^r$. 
If $H^+_r$ equals the hitting time of $[r,\infty)$ by the canonical process we then see that for any $A\in\F_t$
\begin{linenomath}\begin{esn}
\imf{\p_x^r}{A, H^+_r>t}=\imf{\p_x^{r'}}{A, H^+_r>t}
\end{esn}\end{linenomath}
which implies, 
through a projective limit theorem such as Theorem 3.2 in \cite[Ch. V, p. 139]{MR2169627}, 
the existence of a probability measure $\p_x$, 
on \cadlag\ trajectories $\fun{f}{[0,\infty)}{\re\cup\set{\infty}}$ 
for which there exists $\imf{\zeta}{f}\in [0,\infty]$ such that $f(t)=\infty$ if and only if $t\geq \imf{\zeta}{f}$
and such that $\lim_{s\uparrow t}\imf{f}{s}=\infty$ implies $\zeta(f)\leq t$, 
such that
\begin{linenomath}\begin{esn}
\imf{\p_x}{A, H^+_r>t}=\lim_{r'\to\infty }\imf{\p_x^{r'}}{A, H^+_r>t}.
\end{esn}\end{linenomath}This law is then easily seen to be subMarkovian. 
To see that $\p_0$ is the law of a L\'evy process, note that
\begin{linenomath}\begin{esn}
\imf{\p_x^r}{\imf{F}{y+X}}=\imf{\p_{x+y}^{r-y}}{\imf{F}{X}}.
\end{esn}\end{linenomath}
If $F$ is $\F_t$-measurable and bounded and vanishes if $H^+_r\leq t$, 
we can take limits as $r\to\infty$ to obtain
\begin{linenomath}\begin{esn}
\imf{\se_x}{\imf{F}{y+X}}=\imf{\se_{x+y}}{\imf{F}{X}}
\end{esn}\end{linenomath}which implies that the family of laws $\paren{\p_x,x\in\re}$ is spatially homogeneous and so $\p_0$ is the law of a (possibly killed) L\'evy process. 
The L\'evy process has to be spectrally positive since contour processes have only positive jumps. 
Let $\Psi$ be its Laplace exponent. 
We now assert that $\p_x^r$ has the law of a L\'evy process with Laplace exponent $\Psi$ reflected below $r$, 
following the proof of Proposition \ref{ConsistencyUnderTimeChangeProposition}. 
Indeed, recall that under $\p_x^r$ the canonical process is Markov. 
Also, if we stop the canonical process when it surpasses $r$, this stopped process has the same law under $\p_x^r$ and under $\p_x$ by construction. 
Finally, note that under $\p_x$, the level set at $r$ $\set{t\geq 0: X_t=r}$ has Lebesgue measure zero and by truncating at $r'>r$ and time changing, we see that the same holds under $\p_x^r$. 
Hence, $\p_x^r$ coincides with the reflection of $\p_x$ at level $r$. 
We conclude that $\eta=\eta^\Psi$. 
Note that $\Psi$ is then supercritical since otherwise $\eta^\Psi$ would be concentrated on compact TOM trees and $\eta$ was assumed to charge non-compact trees. 
\end{proof}

\appendix

\section{Preliminaries on TOM trees and a Proof of Theorem \ref{CRTCodingTheorem}}
\label{CRTCodingTheoremProof}
The objective of this section is to collect some technical results on TOM trees in Subsection \ref{PreliminaryResultsOnTreesSubsection} which will enable us to construct a coding function for them in Subsection \ref{codingFunctionSubsection} and to prove Theorem \ref{CRTCodingTheorem}. 
\subsection{Preliminary results}
\label{PreliminaryResultsOnTreesSubsection}
We start with an analysis of the greatest common ancestor operator denoted $\wedge$ in Section \ref{IntroductionSection} (page \pageref{greaterCommonAncestorDefinition}). Recall that $\sigma_1\wedge\sigma_2$ is characterized by\begin{linenomath}\begin{esn}
[\rho,\sigma_1]\cap[\rho,\sigma_2]=[\rho,\sigma_1\wedge \sigma_2].
\end{esn}\end{linenomath}Since on the interval $[\sigma_1,\sigma_2]$ the distance $d$ coincides with the usual distance on an interval, for every $\sigma\in[\sigma_1,\sigma_2]$ we have:\begin{equation}
\label{DistanceByDistancesToRootEquation}
\imf{d}{\sigma_1,\sigma_2}=\imf{d}{\sigma_1,\sigma}+\imf{d}{\sigma,\sigma_2}. 
\end{equation}As usual, the open ball of radius $\eps$ centered at $\sigma\in\tau$ is denoted $\ball{\eps}{\sigma}$.
\begin{lemma}
\label{BicontinuityOfWedgeOperatorLemma}
The $\wedge$ operator is bicontinuous. 
\end{lemma}
\begin{proof}
Let $\sigma\in\tau$; for every $\eps\in (0,\imf{d}{\rho,\sigma})$, let $\sigma^\eps$ be the unique point on $[\rho,\sigma]$ at distance $\eps$ from $\sigma$, so that\begin{linenomath}\begin{esn}(\sigma^\eps,\sigma]\in\ball{\eps}{\sigma}
\quad\text{and}\quad [\rho,\sigma^\eps]\not\in\ball{\eps}{\sigma}
. 
\end{esn}\end{linenomath}Note that if $\tilde\sigma\in \ball{\eps}{\sigma}$, then, by considering the unique isometry to an interval, $[\sigma,\tilde\sigma]\subset\ball{\eps}{\sigma}$ and $\sigma^\eps\prec\sigma\wedge\tilde\sigma$. Hence\begin{equation}
\label{SetInclusionForBicontinuityOfWedgeOperator}
[\rho,\sigma^\eps]
\subset 
[\rho,\tilde\sigma]\subset [\rho,\sigma^\eps]\cup\ball{\eps}{\sigma}.
\end{equation}
 
There are 3 different geometries to consider to prove that $\wedge$ is continuous at $\sigma_1,\sigma_2$, using the fact that $\wedge$ is symmetric:
\begin{enumerate}
\item $\sigma_1\wedge\sigma_2\neq\sigma_1,\sigma_2$:  In this case, we choose\begin{linenomath}\begin{esn}
0<\eps<\imf{d}{\sigma_1,\sigma_1\wedge\sigma_2}\wedge\imf{d}{\sigma_1\wedge\sigma_2,\sigma_2}\leq \imf{d}{\sigma_1,\sigma_2}/2
\end{esn}\end{linenomath}where the last inequality holds by equation \eqref{DistanceByDistancesToRootEquation}. 
Note then that\begin{linenomath}\begin{esn}
\sigma_1^\eps\wedge\sigma_2^\eps=\sigma_1\wedge\sigma_2\quad\text{and}\quad \ball{\eps}{\sigma_1}\cap\ball{\eps}{\sigma_2}=\emptyset
\end{esn}\end{linenomath}by choice of $\eps$. Equation \eqref{SetInclusionForBicontinuityOfWedgeOperator} implies that if $\tilde\sigma_i\in\ball{\eps}{\sigma_i}$ then\begin{linenomath}\begin{esn}
[\rho,\sigma_1\wedge\sigma_2]\subset[\rho,\tilde\sigma_1]\cap[\rho,\tilde\sigma_2]\subset [\rho,\sigma_1\wedge\sigma_2]
\end{esn}\end{linenomath}so that $\tilde\sigma_1\wedge\tilde\sigma_2=\sigma_1\wedge\sigma_2$.
\item $\sigma_1\wedge\sigma_2=\sigma_1=\sigma_2$: We write $\sigma$ instead of $\sigma_i$ and consider\begin{linenomath}\begin{esn}
0<\eps<\imf{d}{\rho,\sigma}). 
\end{esn}\end{linenomath}We then note that if $\tilde \sigma_1,\tilde \sigma_2\in\ball{\eps}{\sigma}$ then, by Equation \eqref{SetInclusionForBicontinuityOfWedgeOperator},\begin{linenomath}\begin{esn}
[\rho,\sigma^\eps]\subset [\rho,\tilde\sigma_1]\cap[\rho,\tilde\sigma_2]\subset[\rho,\sigma^\eps]\cup \ball{\eps}{\sigma},
\end{esn}\end{linenomath}so that\begin{linenomath}\begin{esn}
\tilde\sigma_1\wedge\tilde\sigma_2\in\set{\sigma^\eps}\cup\ball{\eps}{\sigma}\subset\ball{2\eps}{\sigma}.
\end{esn}\end{linenomath}
\item $\sigma_1\wedge\sigma_2=\sigma_1\prec\sigma_2$: We consider\begin{linenomath}\begin{esn}
0<\eps<\frac{1}{2}\,\imf{d}{\rho,\sigma_1}\wedge\imf{d}{\sigma_1,\sigma_2},
\end{esn}\end{linenomath}so that\begin{linenomath}\begin{esn}
\ball{\eps}{\sigma_1}\cap\ball{\eps}{\sigma_2}=\emptyset\quad\text{and}\quad\sigma_1^\eps\prec\sigma_2^\eps. 
\end{esn}\end{linenomath}If $\tilde\sigma_i\in\ball{\eps}{\sigma_i}$ then Equation \eqref{SetInclusionForBicontinuityOfWedgeOperator} gives\begin{linenomath}\begin{esn}
[\rho,\sigma_1^\eps]\subset [\rho,\tilde\sigma_1]\cap[\rho,\tilde\sigma_2]\subset [\rho,\sigma_1^\eps]\cup\ball{\eps}{\sigma_1}
\end{esn}\end{linenomath}so that $\tilde\sigma_1\wedge\tilde\sigma_2\in\ball{2\eps}{\sigma_1}$.\qedhere
\end{enumerate}
\end{proof}
%
%

We tacitly assumed the sets involved in \defin{Mes1} of the definition of a TOM tree are measurable. 
They are actually closed as we now show. For any $\sigma\in\tau$, define the \defin{left, strict left, right, and strict right} of $\sigma$ as follows:\begin{linenomath}\begin{align*}
&L_\sigma=\set{\tilde\sigma\in\tau:\tilde\sigma\leq \sigma},
&&L_{\sigma-}=\set{\tilde\sigma\in\tau:\tilde\sigma< \sigma},\\
&R_\sigma=\set{\tilde\sigma\in\tau:\tilde\sigma\geq \sigma},
&&R_{\sigma+}=\set{\tilde\sigma\in\tau:\tilde\sigma> \sigma}.
\end{align*}\end{linenomath}
\begin{lemma}
\label{StrictLeftIsOpenLemma}
For any $\sigma\in\tau$, the strict left of $\sigma$ is open.
\end{lemma}
\begin{proof}
Let $\sigma_1<\sigma_2$ and consider $\delta=\imf{d}{\sigma_1,\sigma_1\wedge\sigma_2}$. Since $\sigma_1<\sigma_2$, $\delta>0$ since if $\sigma_1\wedge\sigma_2$ were equal to one of the $\sigma_i$, it would have to be $\sigma_2$ thanks to \defin{Or1}. We now show that $\ball{\delta}{\sigma_1}\subset L_{\sigma_2-}$. Let $\sigma\in\ball{\delta}{\sigma_1}$. 

Note that $\sigma\wedge\sigma_1\in\ball{\delta}{\sigma_1}$ since\begin{linenomath}\begin{esn}
\imf{d}{\sigma,\sigma\wedge\sigma_1}+\imf{d}{\sigma\wedge\sigma_1,\sigma_1}=\imf{d}{\sigma,\sigma_1}<\delta
\end{esn}\end{linenomath}Using the isometry $\phi_{\rho,\sigma_1}$, we see that\begin{linenomath}\begin{align*}
\imf{d}{\rho,\sigma\wedge\sigma_1}
&=\imf{d}{\rho,\sigma_1}-\imf{d}{\sigma_1,\sigma\wedge\sigma_1}
\\&=\imf{d}{\rho,\sigma_1\wedge\sigma_2}+\imf{d}{\sigma_1\wedge\sigma_2,\sigma_1}-\imf{d}{\sigma_1,\sigma\wedge\sigma_1}
\\&=\imf{d}{\rho,\sigma_1\wedge\sigma_2}+\delta-\imf{d}{\sigma_1,\sigma\wedge\sigma_1}
\\&>\imf{d}{\rho,\sigma_1\wedge\sigma_2}
\end{align*}\end{linenomath}we obtain\begin{linenomath}\begin{esn}
\sigma_1\wedge \sigma\in (\sigma_1\wedge\sigma_2,\sigma_1];
\end{esn}\end{linenomath}property \defin{Or2} then implies  $\sigma_1\wedge\sigma<\sigma_2$ and then, using \defin{Or1}, we get\begin{linenomath}\begin{esn}
\sigma\leq \sigma\wedge\sigma_1<\sigma_2.\qedhere
\end{esn}\end{linenomath}
\end{proof}
Note however that the strict right is in general not open. Indeed, if $\sigma_1<\sigma_2$ and $\sigma_1,\sigma_2\neq\sigma_1\wedge\sigma_2$, then $\sigma_1\wedge\sigma_2\in R_{\sigma_2+}$ but every neighborhood of $\sigma_1\wedge\sigma_2$ intersects $[\sigma_1,\sigma_1\wedge\sigma_2)\subset L_{\sigma_2}$. (In the context of the above proof, $\delta$ could be $0$.) 

The previous lemma ensures that right sets are closed, as complementary sets of strict left sets. 
To prove that left (and so strict right) sets are measurable, we use Lemma \ref{StrictLeftIsOpenLemma} in conjunction with Lemma \ref{ThreeOrderedPointsLemma} below.

Consider the set $\tau_\sigma=\set{\sigma'\in\tau: \sigma\preceq \sigma'}$. 
\begin{lemma}
\label{ThreeOrderedPointsLemma}
If $\sigma_1\leq \sigma_2\leq \sigma_3$ then $\sigma_1\wedge\sigma_2\in [\sigma_1,\sigma_1\wedge\sigma_3]$ and $\sigma_2\in\tau_{\sigma_1\wedge\sigma_3}$. 
\end{lemma}
\begin{proof}
It suffices to prove the result when inequalities are strict. 
Note that the assumption $\sigma_1\wedge\sigma_2\in [\rho,\sigma_1\wedge\sigma_3)$ leaves us with the impossible cases\begin{description}
\item[$\sigma_2>\sigma_1\wedge\sigma_3$] since $\sigma_1\wedge\sigma_3\geq \sigma_3$ by \defin{Or1} and $\sigma_3>\sigma_2$ and 
\item[$\sigma_2\leq \sigma_1\wedge\sigma_3$] since \defin{Or2} and $\sigma_1< \sigma_2$ imply $\sigma_1\wedge\sigma_3\in [\sigma_1,\sigma_1\wedge\sigma_2)<\sigma_2$ which implies $\sigma_1\wedge\sigma_3<\sigma_2$.
\end{description}

Finally, since $\sigma_1\wedge\sigma_3\preceq \sigma_1\wedge\sigma_2\preceq\sigma_2$, we see that $\sigma_2\in\tau_{\sigma_1\wedge\sigma_3}$.
\end{proof}

To prove that $L_\sigma$ is measurable, note first that this is true when $\sigma=\rho$. 
When $\sigma\neq\rho$, choose $\sigma_n\in [\rho,\sigma)$ converging to $\sigma$. 
We assert that\begin{linenomath}\begin{esn}
L_\sigma=\cap_n L_{\sigma_n-},
\end{esn}\end{linenomath}where the right-hand side is measurable by Lemma \ref{StrictLeftIsOpenLemma}. Indeed, $\tilde\sigma\leq\sigma$ implies $\tilde\sigma<\sigma_n$ for all $n$. 
On the other hand, if $\tilde\sigma<\sigma_n$ for all $n$  and $\sigma<\tilde\sigma$, then Lemma \ref{ThreeOrderedPointsLemma} gives $\tilde\sigma\in\tau_{\sigma\wedge\sigma_n}$. 
However, Lemma \ref{BicontinuityOfWedgeOperatorLemma} tells us that $\sigma\wedge\sigma_n\to\sigma$, so that $\tilde\sigma\in \tau_{\sigma}$ which contradicts $\sigma<\tilde\sigma$.

We now give a simple sufficient condition for a sequence to converge.
\begin{lemma}
\label{MonotonicSequencesConvergeLemma}
Any $\leq$-monotonic sequence on $\tau$ converges.
\end{lemma}
\begin{proof}
Let $\sigma_1,\sigma_2,\ldots$ a $\leq$-monotonic sequence. 
Since $\tau$ is compact, we must only show that all convergent subsequences have the same limit. 

Suppose there exist two subsequences $\sigma_{n_k^1}$ and $\sigma_{n_k^2}$ converging to $\sigma^1$ and $\sigma^2$ where $\sigma^1<\sigma^2$. 
Since $\sigma^1<\sigma^2$, it follows that $\sigma^1\neq \sigma^1\wedge\sigma^2$, so that we can choose $0<\delta<\imf{d}{\sigma^1,\sigma^1\wedge \sigma^2}$.

Let $\sigma^{1,\delta}$ be the element of $[\sigma^1,\rho]$ at distance $\delta$ from $\sigma^1$; in terms of isometries between intervals and paths on the tree, $\sigma^{1,\delta}=\imf{\phi_{\sigma^1,\rho}}{\delta}$. 
Note that for $\sigma,\tilde\sigma\in\ball{\delta}{\sigma^1}$, since $\sigma\wedge\sigma^1\in\ball{\delta}{\sigma^1}$ and similarly for $\tilde\sigma$, then\begin{linenomath}\begin{esn}
[\rho, \sigma^{1,\delta}]\subsetneq [\rho,\sigma\wedge\sigma_1]\cap[\rho,\tilde\sigma\wedge\sigma_1]\subset [\rho,\sigma]\cap[\rho,\tilde\sigma]=[\rho,\sigma\wedge\tilde\sigma]
\end{esn}\end{linenomath} and so $\sigma^{1,\delta}\prec\sigma\wedge\tilde\sigma$.

Let $K\geq 1$ be such that for $k\geq K$, $\imf{d}{\sigma_{n_k^1},\sigma^1}<\delta$. 
For $n\geq n_K^1$ and $k$ large enough (depending on $n$)
\begin{linenomath}\begin{esn}
\sigma_{n_K^1}\leq \sigma_n\leq \sigma_{n_k^1}\quad\text{or}\quad \sigma_{n_K^1}\geq \sigma_n\geq \sigma_{n_k^1}
\end{esn}\end{linenomath}Lemma \ref{ThreeOrderedPointsLemma} implies\begin{linenomath}\begin{esn}
\sigma_n\in \tau_{\sigma_{n_K^1}\wedge \sigma_{n_k^1}}\subset \tau_{\sigma^{1,\delta}},
\end{esn}\end{linenomath}the last inclusion holding because of the preceding paragraph.

Hence $\sigma_1\wedge\sigma_2\prec \imf{\phi_{\sigma^1,\rho}}{\delta}\prec\sigma_n$ so that\begin{linenomath}\begin{esn}
\imf{d}{\sigma_n,\sigma^2}\geq \imf{d}{\sigma^{1,\delta},\sigma_1\wedge\sigma_2}>0
\end{esn}\end{linenomath}and so $\sigma_{n_k^2}$ cannot converge to $\sigma^2$.
\end{proof}

In order to characterize measures on TOM trees, note that the set\begin{linenomath}\begin{esn}
\mc{L}=\set{L_\sigma:\sigma\in\tau}\cup\set{\emptyset}
\end{esn}\end{linenomath}is a $\pi$-system (since $\leq$ is a total order). 
Let us define $\mc{A}=\sag{\mc{L}}$.
\begin{lemma}
\label{PiSystemLemma}
The set $\mc{L}$ generates the Borel subsets of $\tau$. 
\end{lemma}
\begin{proof}

Note that $\clo{\ball{\delta}{\sigma}}=\set{\tilde\sigma:\imf{d}{\tilde\sigma,\sigma}\leq \delta}$ and denote the latter set by $C$. 
If suffices to show that $C\in\mc{A}$ (for every $\sigma\in\tau$ and $\delta>0$). 
Since $\tau\setminus C$ is open, it is the union of a countable set of open components, say $C_1,C_2,\ldots$. 
Note that real trees are locally path connected (one way to see this is because closed balls of a real tree are real trees themselves and another one is to prove directly that if $\sigma_1,\sigma_2\in\ball{\delta}{\sigma}$ then $[\sigma_1,\sigma_2]\subset\ball{\delta}{\sigma}$). 
Hence the notions of path-wise connectedness and connectedness coincide for real trees.

We assert that if $\sigma_1,\sigma_2\in C_i$ and $\sigma_1\leq\tilde\sigma\leq\sigma_2$ then $\tilde \sigma\in C_i$. 
Indeed, note first that $[\sigma_1,\sigma_2]\subset C_i$: 
this follows since $[\sigma_1,\sigma_2]\subset \tau\setminus C$ (or $\sigma_1$ and $\sigma_2$ would belong to different path-connected components of $\tau\setminus C$) and then since $[\sigma_1,\sigma_2]$ is connected and has non-empty intersection with $C_i$, it follows that $C_i\cup [\sigma_1,\sigma_2]$ is connected and by definition of connected component, $[\sigma_1,\sigma_2]\subset C_i$. 
If $\sigma_1<\tilde\sigma<\sigma_2$,
Lemma \ref{ThreeOrderedPointsLemma} implies $\tilde\sigma\wedge\sigma_1\in [\sigma_1,\sigma_1\wedge\sigma_2]$.  
Hence $C_i\cup [\sigma_1\wedge\tilde\sigma,\tilde\sigma]$ is connected so that $[\sigma_1\wedge\tilde\sigma,\tilde\sigma]\subset C_i$ and $\tilde\sigma\in C_i$. 

Let\begin{linenomath}\begin{esn}
L_i=\set{L_{\tilde\sigma}:\tilde\sigma\in C_i},\quad I_i=\inf L_i\quad\text{and}\quad S_i=\sup L_i.
\end{esn}\end{linenomath}If $I_i\in L_i$, there exists a $\leq$-smallest element $\underline\sigma$ of $C_i$, and likewise if $S_i\in L_i$ there exists a $\leq$-largest element $\overline \sigma$ of $C_i$ which would imply\begin{linenomath}\begin{esn}
C_i=L_{\overline\sigma}\setminus L_{\underline\sigma-}\in\mc{A}.
\end{esn}\end{linenomath}If $I_i\not\in L_i$ but $S_i\in L_i$, consider a decreasing sequence $\underline\sigma_n$ such that $\imf{\mu}{L_{\underline\sigma_n}}\to I_i$ and note that in this case\begin{linenomath}\begin{esn}
C_i=\bigcup_{n} L_{\overline\sigma}\setminus L_{\sigma_n}\in\mc{A}.
\end{esn}\end{linenomath}Remaining cases are handled similarly, which shows that\begin{linenomath}\begin{esn}
\tau\setminus C=\bigcup C_i\in\mc{A},
\end{esn}\end{linenomath}
which terminates the proof.
\end{proof}
Hence, any measure on the Borel sets of $\tau$ is determined by its values on $\mc{L}$.

\subsection{The contour process of a compact TOM tree}
\label{codingFunctionSubsection}
The reason for introducing compact TOM trees is that they provide us with a simple way to code their elements by real numbers. 
This is formally done through exploration process and the height function now defined. 
First, define $\fun{\psi}{\tau}{[0,\imf{\mu}{\tau}]}$ by\begin{linenomath}\begin{esn}
\imf{\psi}{\sigma}=\imf{\mu}{L_\sigma},
\end{esn}\end{linenomath}where we recall that $L_\sigma$ is the left of $\sigma$. 
Note that $\psi$ is strictly increasing thanks to \defin{Mes}. 
\begin{lemma}
\label{DensityOfImageLemma}
The set $D=\imf{\psi}{\tau}$ is dense in $[0,\imf{\mu}{\tau}]$.
\end{lemma}
\begin{remark}
If a tree does not have a first element, then $0\not\in D$; an actual example can be seen in Figure \ref{treeWithoutFirstElement}. 
\end{remark}
\begin{figure}
\includegraphics[width=.45\textwidth,height=.15\textheight]{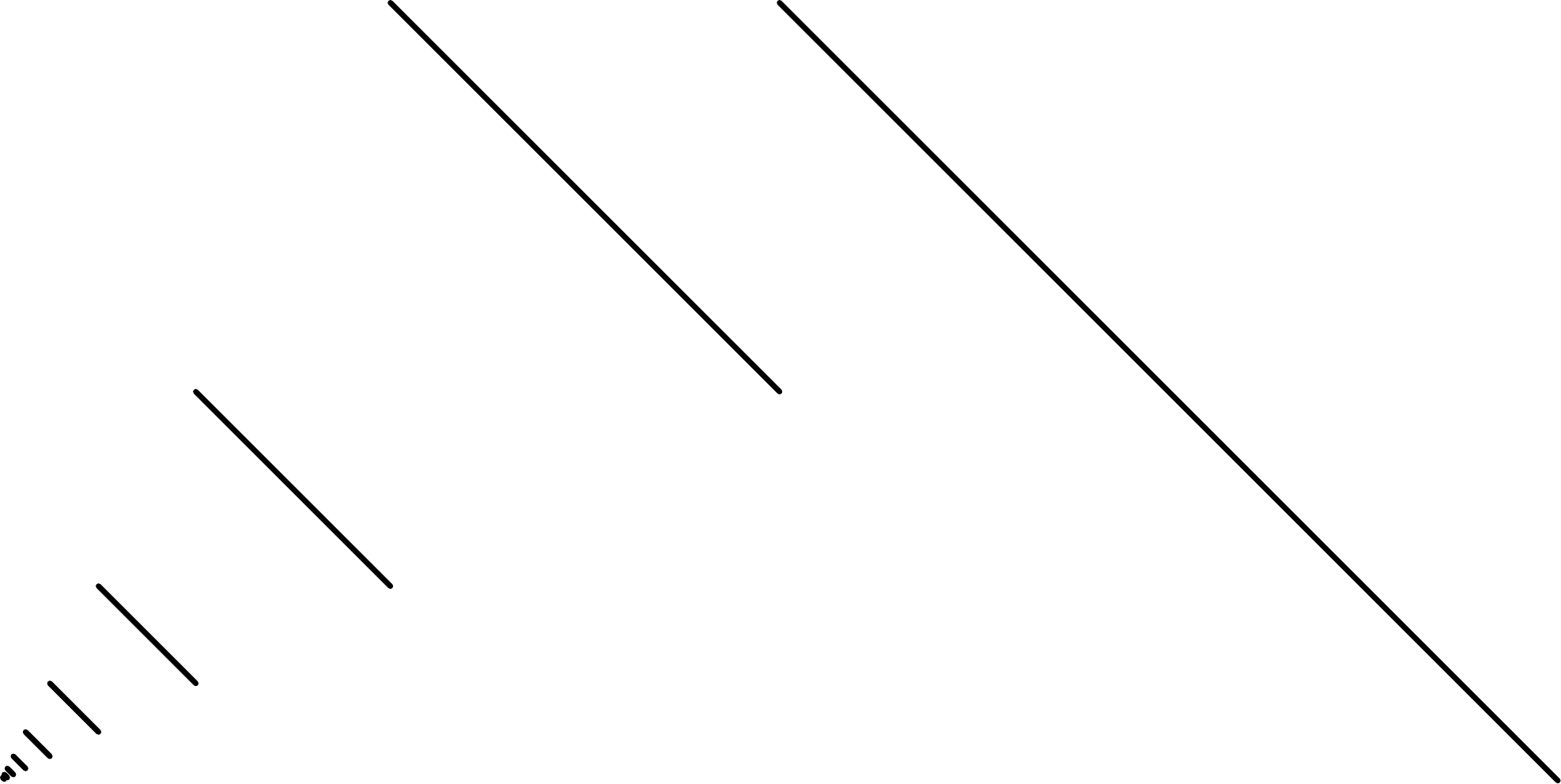}
\hfill
\includegraphics[width=.45\textwidth,height=.15\textheight]{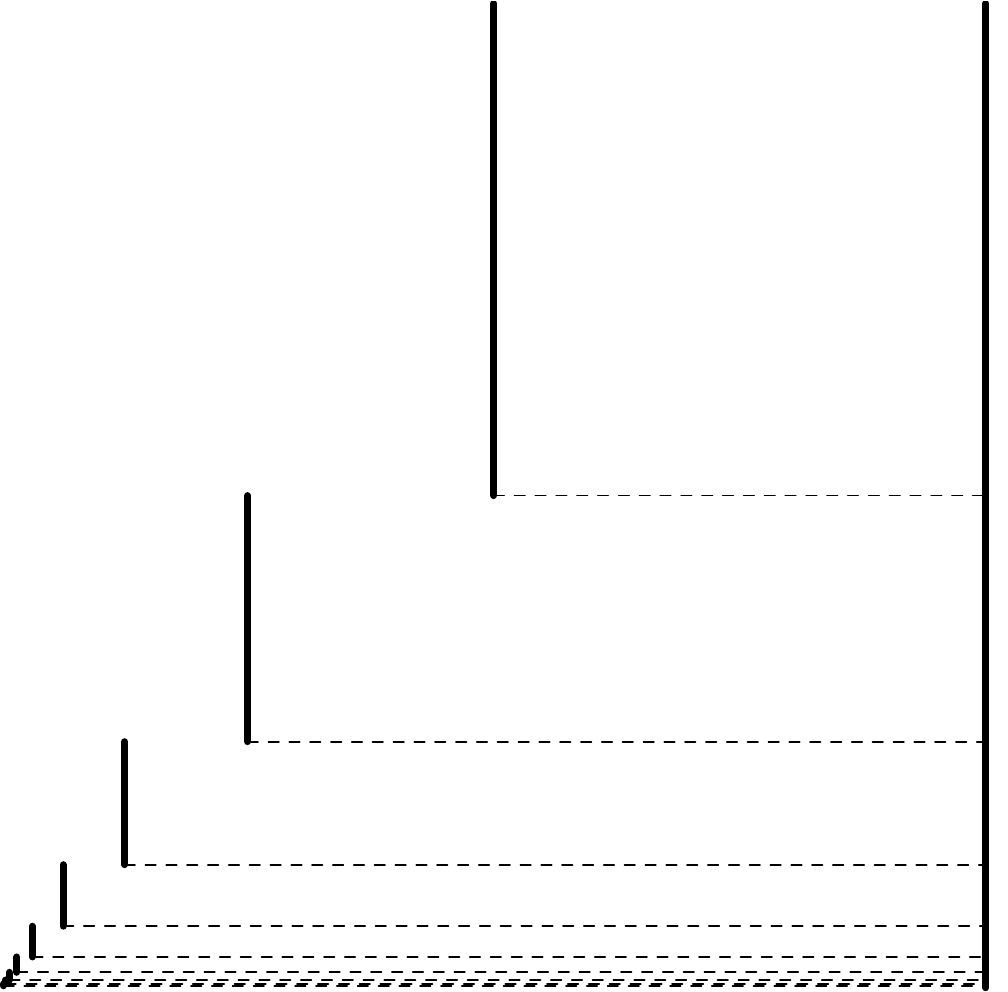}
\caption{A \cadlag\ function coding a TOM tree without a first element (visualized on the right through  vertical segments joined via dashed lines).}
\label{treeWithoutFirstElement}
\end{figure}
\begin{proof}
Let $t\in(0,\imf{\mu}{\tau})$. 
Set\begin{linenomath}\begin{esn}
G_t=\set{\sigma\in\tau: \imf{\psi}{\sigma}<t}
\text{ and }
D_t=\set{\sigma\in\tau: \imf{\psi}{\sigma}\geq t}.
\end{esn}\end{linenomath}Also set
\begin{linenomath}\begin{esn}
s_t=\sup\set{\imf{\psi}{\sigma}:\sigma\in G_t}\quad\text{ and }
i_t=\inf\set{\imf{\psi}{\sigma}:\sigma\in D_t},
\end{esn}\end{linenomath}so that in particular $s_t\leq t\leq i_t$. 

First, notice that for any $\sigma\in G_t$ we have the inclusion $L_\sigma\subset G_t$, so that $G_t$ is necessarily of the form $L_\sigma$ or $L_\sigma\setminus\set{\sigma}$, which yields $\imf{\mu}{G_t}=s_t$. 

Now, by definition of $i_t$, there is some $\leq$-decreasing sequence $(\sigma_n)$ of elements of $D_t$ such that $\imf{\mu}{\sigma_n}\downarrow i_t$. 
Since $(\sigma_n)$ is decreasing, the sequence $L_{\sigma_n}$ is also decreasing; let $L$ denote its limit. 
If there were two elements in $L\setminus G_t$, say $\sigma_1<\sigma_2$, we would have $t\leq \imf{\psi}{\sigma_1}<\imf{\psi}{\sigma_2}\leq i_t$ by \defin{Mes1}. 
Now this contradicts the definition of $i_t$ since $L\setminus G_t\subset D_t$, so $L\setminus G_t$ contains at most one element. 
By \defin{Mes2}, this shows that $\imf{\mu}{L}=\imf{\mu}{G_t}$. 
Now recall that $\imf{\mu}{G_t}=s_t$, so that
\begin{linenomath}\begin{esn}
i_t=\lim_n \imf{\psi}{\sigma_n}=\lim_n \imf{\mu}{L_{\sigma_n}}=\imf{\mu}{L}=\imf{\mu}{G_t}=s_t, 
\end{esn}\end{linenomath}which shows that $i_t=s_t=t$. 
\end{proof}

We can now define the \defin{exploration process} $\fun{\phi}{[0,\imf{\mu}{\tau}]}{\tau}$ by means of\begin{linenomath}\begin{esn}
\imf{\phi}{t}=\lim_{\substack{s\to t+\\ s\in D}}\imi{s}{\psi};
\end{esn}\end{linenomath}notice that $\imi{s}{\psi}$ is $\leq$-decreasing  if $s$ decreases. Therefore, if $s_n\downarrow t$ along $D$, then $\imi{s_n}{\psi}$ converges thanks to Lemma \ref{MonotonicSequencesConvergeLemma}. 
To see that the limit does not depend on the subsequence, notice that if $s_n$ and $\tilde s_n$ both decrease to $t$ along $D$, they can be intertwined into a decreasing sequence which contains them both as subsequences. 

\begin{lemma}
\label{ExplorationProcessLemma}
The exploration process is the unique \cadlag\ extension of $\psi^{-1}$ to $[0,\imf{\mu}{\tau}]$. 
Furthermore, it satisfies the following property: for any $t_1\in [0,\imf{\mu}{\tau}]$,\begin{linenomath}\begin{esn}
\imf{\phi}{\left[t_1,\imf{\psi}{\imf{\phi}{t_1}}\right]}\subset \tau_{\imf{\phi}{t_1}}.
\end{esn}\end{linenomath}
\end{lemma}
\begin{proof}
Note first that $\phi=\psi^{-1}$ on $D$. 
Indeed, if $t\in D$, say $t=\imf{\mu}{L_\sigma}$, and $t_n\in D$ decreases to $t$, we let $\sigma_n=\imi{t_n}{\psi}\geq \sigma$. 
Since $\sigma_n$ is decreasing it converges thanks to Lemma \ref{MonotonicSequencesConvergeLemma}. 
Call the limit $\tilde \sigma$. 
The inequality $\tilde \sigma<\sigma$ is impossible since $\tilde\sigma<\sigma\leq \sigma_n$ and  Lemma \ref{StrictLeftIsOpenLemma} imply that $\sigma_n\not\to \tilde\sigma$. 
Bu the same argument, $\tilde\sigma\leq \sigma_n$ for all $n$. 
Then, as in \ref{DensityOfImageLemma}, we note that $\bigcap_n L_{\sigma_n}\setminus L_{\tilde \sigma}$ has at most one element, since if it contains an element $\tilde\tilde\sigma>\tilde \sigma$, then $\sigma_n\not\to\sigma$. We conclude that $\imf{\mu}{L_{\sigma_n}}\to\imf{\mu}{L_\sigma}$. 
Then, the inequality $\sigma<\tilde \sigma$ implies the contradiction\begin{linenomath}\begin{esn}
t=\lim_n t_n=\lim_n\imf{\mu}{L_{\sigma_n}}=\imf{\mu}{L_{\tilde \sigma}}>\imf{\mu}{L_\sigma}=t
\end{esn}\end{linenomath}by use of \defin{Mes1}, so that $\sigma_n$ decreases to $\sigma$ and so $\imf{\phi}{t}=\imi{t}{\psi}$.

We now prove that $\phi$ is right-continuous. 
Let $t<\imf{\mu}{\tau}$ and $t_n\downarrow t$ with $t_n>t$. 
Then there exists $\tilde t_n\in D$ such that\begin{linenomath}\begin{esn}
t<t_n<\tilde t_n, \quad\tilde t_n-t_n\to 0
\quad\text{ and }
\quad\imf{d}{\imf{\phi}{t_n},\imi{\tilde t_n}{\psi}}\to 0. 
\end{esn}\end{linenomath}By monotonicity of $\psi^{-1}$ on $D$, we see that\begin{linenomath}\begin{esn}
\imf{\phi}{t}=\lim_{n\to\infty}\imi{\tilde t_n}{\psi}=\lim_{n\to\infty}\imf{\phi}{t_n}.
\end{esn}\end{linenomath}

A similar argument shows that $\phi$ has left limits: note first that by monotonicity the limit\begin{linenomath}\begin{esn}
\lim_{\substack{s\to t-\\s\in D}}\imi{s}{\psi}
\end{esn}\end{linenomath}exists. 
If $t_n\uparrow t$ with $t_n<t$, there exist $\tilde t_n\in D$ such that\begin{linenomath}\begin{esn}
t_n<\tilde t_n<t,\quad \tilde t_n-t_n\to 0\quad\text{and}\quad \imf{d}{\imf{\phi}{t_n},\imi{\tilde t_n}{\psi}}\to 0.
\end{esn}\end{linenomath}By monotonicity we see that\begin{linenomath}\begin{esn}
\lim_{\substack{s\to t-\\s\in D}}\imi{s}{\psi}=\lim_{n} \imi{t_n}{\psi},
\end{esn}\end{linenomath}so that\begin{linenomath}\begin{esn}
\imf{\phi}{t-}=\lim_{\substack{s\to t-\\s\in D}}\imi{s}{\psi}.
\end{esn}\end{linenomath}
Uniqueness of the extension follows from density of $D$.

We finally prove that\begin{linenomath}\begin{esn}
\imf{\phi}{[t_1,\imf{\mu}{L_{\imf{\phi}{t_1}}}]}\subset \tau_{\imf{\phi}{t_1}}.
\end{esn}\end{linenomath}Let us then consider\begin{linenomath}\begin{esn}
t_1<t<\imf{\mu}{L_{\imf{\phi}{t_1}}}.
\end{esn}\end{linenomath}Take $t^n_1\in D$ decreasing to $t_1$ and $t^n\in D$ decreasing to $t$ and let $\sigma^n_1=\imf{\phi}{t^n_1}$ and $\sigma^n=\imf{\phi}{t^n}$. 
Then, by definition, $\sigma^n_1\to\imf{\phi}{t_1}$ and $\sigma^n\to\imf{\phi}{t}$. 
Since $\psi^{-1}$ is order preserving and coincides with $\phi$ on $D$, we see that\begin{linenomath}\begin{esn}
\sigma^n_1\leq \sigma^n\leq \imf{\phi}{t_1}.
\end{esn}\end{linenomath}Lemma \ref{ThreeOrderedPointsLemma} implies
\begin{linenomath}\begin{esn}
\sigma^n\wedge \sigma_1^n\in [\sigma_1^n,\sigma_1^n\wedge \imf{\phi}{t_1}].
\end{esn}\end{linenomath}
By passing to the limit in the expression\begin{linenomath}\begin{esn}
\sigma_1^n\wedge \imf{\phi}{t_1}\preceq \sigma_1^n\wedge\sigma^n,
\end{esn}\end{linenomath}using Lemma \ref{BicontinuityOfWedgeOperatorLemma}, we get\begin{linenomath}\begin{esn}
\imf{\phi}{t_1}\preceq\imf{\phi}{t_1}\wedge\imf{\phi}{t},
\end{esn}\end{linenomath}
which ends the proof.
\end{proof}

We now define the \defin{contour process} $\fun{f}{[0,\imf{\mu}{\tau}]}{\tau}$ of $\tr{c}=\paren{\paren{\tau,d,\rho},\leq,\mu}$ by\begin{linenomath}\begin{esn}
f_t
=\imf{d}{\rho,\imf{\phi}{t}}.
\end{esn}\end{linenomath}

\begin{proposition}
\label{HeightProcessCodesTreeProposition}
The contour process $f$ is \cadlag\ with no negative jumps and the compact TOM tree coded by $f$ is isomorphic to $\tr{c}$.
\end{proposition}
\begin{proof}
While it is clear that $f$ is \cadlag, it is less evident that its jumps are non-negative. 
However, if $\imf{f}{t}<\imf{f}{t-}$ then $\imf{\phi}{t-}\not\in [\rho,\imf{\phi}{t}]$. 
We obtain a contradiction by analyzing the cases $\imf{\phi}{t-}<\imf{\phi}{t}$ or  $\imf{\phi}{t}<\imf{\phi}{t-}$. 
In the first case, any element of $(\imf{\phi}{t-},\imf{\phi}{t-}\wedge \imf{\phi}{t})$ has a measure exceeding $t$ and is $< \imf{\phi}{t}$. 
Indeed let $\sigma\in (\imf{\phi}{t-},\imf{\phi}{t-}\wedge \imf{\phi}{t})$ and suppose that $t_n\in D$ increase to $t$ so that $\imf{\phi}{t_n}\to\imf{\phi}{t-}$. 
By Lemma \ref{BicontinuityOfWedgeOperatorLemma}\begin{linenomath}\begin{esn}
\imf{\phi}{t_n}\wedge \imf{\phi}{t-}\to\imf{\phi}{t-}
\end{esn}\end{linenomath}and so for large enough $n$, $\imf{\phi}{t_n}\leq \imf{\phi}{t_n}\wedge \imf{\phi}{t-}<\sigma$ so that\begin{linenomath}\begin{esn}
t_n=\imf{\mu}{L_{\imf{\phi}{t_n}}}<\imf{\mu}{L_\sigma}.
\end{esn}\end{linenomath}In the second case, for any $\sigma\in L_{\imf{\phi}{t}}$, we have the bound
\begin{linenomath}
\begin{esn}
\imf{d}{\imf{\phi}{t-},\sigma}\geq \imf{d}{\imf{\phi}{t-},\imf{\phi}{t}\wedge \imf{\phi}{t-}}
\end{esn}\end{linenomath}(since $\sigma$ cannot belong to $[\imf{\phi}{t}\wedge \imf{\phi}{t-},\imf{\phi}{t-}]$). 
However, $\imf{\phi}{t_n}\in L_{\imf{\phi}{t}}$, which contradicts the fact that $\imf{\phi}{t_n}\to \imf{\phi}{t-}$. 

Let us now analyze the tree coded by $f$. 
Recall that if we define the pseudo-distance\begin{linenomath}\begin{esn}
\imf{d_f}{t_1,t_2}=\imf{f}{t_1}+\imf{f}{t_2}-2\inf_{t\in [t_1,t_2]} \imf{f}{t},
\end{esn}\end{linenomath}then the tree coded by $f$ is the quotient space of $[0,\imf{\mu}{\tau}]$ under the equivalence relation $\sim_f$ given by $t_1\sim_f t_2$ if $\imf{d_f}{t_1,t_2}=0$. 

We first prove that $\phi$ is constant on the equivalence classes of $\sim_f$. 
Indeed, consider $t_1<t_2$ such that $\imf{\phi}{t_1}\neq \imf{\phi}{t_2}$ and let us prove that $t_1\not \sim_f t_2$. 
We have 3 cases: \begin{description}
\item[$\imf{\phi}{t_1}\prec \imf{\phi}{t_2}$] Then $\imf{f}{t_1}<\imf{f}{t_2}$ and so $t_1\not\sim_f t_2$.
\item[$\imf{\phi}{t_2}\prec \imf{\phi}{t_1}$] Then $\imf{f}{t_2}<\imf{f}{t_1}$ and so $t_1\not\sim_f t_2$.
\item[$\imf{\phi}{t_1}\wedge\imf{\phi}{t_2}\neq \imf{\phi}{t_1},\imf{\phi}{t_2}$] Note that $t_1<t_2\leq\imf{\mu}{L_{\imf{\phi}{t_2}}}$. 
We now use Lemma \ref{ExplorationProcessLemma} and the fact that $\imf{\phi}{t_2}\not\in\tau_{\imf{\phi}{t_1}}$ to deduce\begin{linenomath}\begin{esn}
t_1\leq \imf{\mu}{L_{\imf{\phi}{t_1}}}<t_2\leq \imf{\mu}{L_{\imf{\phi}{t_2}}}.
\end{esn}\end{linenomath}Hence $\imf{\phi}{t_1}<\imf{\phi}{t_2}$. 
Let $\sigma\in (\imf{\phi}{t_1},\imf{\phi}{t_1}\wedge\imf{\phi}{t_2})$. 
Then\begin{linenomath}\begin{esn}
t_1\leq \imf{\mu}{L_{\imf{\phi}{t_1}}}<\imf{\mu}{L_\sigma}<\imf{\mu}{L_{\imf{\phi}{t_2}}}. 
\end{esn}\end{linenomath}Also, since $\sigma\not\in \tau_{\imf{\phi}{t_2}}$, then, actually,  $\imf{\mu}{L_\sigma}<t_2$. 
We then have:\begin{linenomath}\begin{esn}
\imf{\mu}{L_\sigma}\in (t_1,t_2)\quad\text{and}\quad \imf{f}{\imf{\mu}{L_\sigma}}=\imf{d}{\rho,\sigma}<\imf{d}{\rho,\imf{\phi}{t_1}}=\imf{f}{t_1},
\end{esn}\end{linenomath}so that $t_1\not\sim_f t_2$.
\end{description}

Abusing notation, we can then define  $\phi$ on the equivalence class $[t]_f$ of $t$ under $\sim_f$ as $\imf{\phi}{t}$. 
$\phi$ is a bijection from the set $\tau_f$ of equivalence classes under $\sim_f$ to $\tau$: it is surjective because $\phi$ is surjective on $[0,\imf{\mu}{\tau}]$, since
\begin{linenomath}\begin{esn}
\imf{\phi}{\imf{\psi}{\sigma}}=\sigma, 
\end{esn}\end{linenomath}and it is injective thanks to the following argument. 
If $\imf{\phi}{t_1}=\imf{\phi}{t_2}$ and $t_1<t_2$, then\begin{linenomath}\begin{esn}
\imf{f}{t_1}=\imf{f}{t_2}\quad\text{and}\quad t_2\leq \imf{\mu}{L_{\imf{\phi}{t_2}}}=\imf{\mu}{L_{\imf{\phi}{t_1}}}. 
\end{esn}\end{linenomath}By Lemma \ref{ExplorationProcessLemma}, we see that $\imf{\phi}{[t_1,t_2]}\subset\tau_{\imf{\phi}{t_1}}$, so that for every $t\in[t_1,t_2]$, $\imf{f}{t}\geq \imf{f}{t_1}$. 
Hence, $t_1\sim_f t_2$.

It remains to see that $\phi$ is an isometry between $\tau_f$ and $\tau$ which preserves root, order and measure. 

Since $\tau=L_\rho$, we see that $\rho=\imf{\phi}{\imf{\mu}{\tau}}$ and that $f$ reaches its minimum (which is zero) at $\imf{\mu}{\tau}$. 
Hence, $\rho_f=[\imf{\mu}{\tau}]_f$ and $\imf{\phi}{\rho_f}=\rho$. 
This implies that for all $t\in\tau_f$:\begin{equation}
\label{DistanceToRootIsPreservedEquation}
\imf{d}{\imf{\phi}{t},\rho}=\imf{f}{t}=\imf{d_f}{t,\imf{\mu}{\tau}}.
\end{equation}

The bijection $\fun{\phi}{\tau_f}{\tau}$ is an isometry by the following observations: 
firstly, on every compact rooted real tree distances between every pair of elements is determined by distances to the root, thanks to 
the formula
\begin{linenomath}\begin{esn}
\imf{d}{\sigma_1,\sigma_2}=\imf{d}{\sigma_1,\rho}+\imf{d}{\sigma_2,\rho}-2\imf{d}{\sigma_1\wedge \sigma_2,\rho}. 
\end{esn}\end{linenomath}Hence, if $\tilde d$ is a metric on $\tau$ and $(\tau,\tilde d,\rho)$ is a real tree, then $\imf{\tilde d}{\sigma,\rho}=\imf{d}{\sigma,\rho}$ implies that $d=\tilde d$. 
Equation \eqref{DistanceToRootIsPreservedEquation} then proves that $\phi$ is an isometry. 

To see that $\phi$ preserves order,  note first that $\sup[t]_f=\mu(L_{\imf{\phi}{t}})$ for every $t$. 
Indeed, if $\mu(L_{\imf{\phi}{t}})<\tilde t$, since $\tilde t\leq \mu(L_{\phi(\tilde t)})$, we see that\begin{linenomath}\begin{esn}
\imf{\mu}{L_{\imf{\phi}{t}}}<\imf{\mu}{L_{\imf{\phi}{\tilde t}}}
\end{esn}\end{linenomath}so that $\imf{\phi}{t}<\imf{\phi}{\tilde t}$ and $\tilde t\not\in [t]_f$. 
Hence, if $t_1,t_2$ are such that $\sup[t_1]_f< \sup[t_2]_f$ then $\mu(L_{\imf{\phi}{t_1}})<\mu(L_{\imf{\phi}{t_2}})$  by \defin{Mes1}, so that $\imf{\phi}{t_1}<\imf{\phi}{t_2}$.

Finally, to see that $\phi$ is measure preserving, recall that $p_f$ stands for the projection of $[0,\imf{\mu}{\tau}]$ into $\tau_f$ (sending $t$ to its equivalence class) defined on page \pageref{ProjectionFromIntervalToRealTreeDefinition} and that $\mu_f=\leb\circ p_f^{-1}$ by definition. 
We need to prove the equality $\mu=\leb\circ p_f^{-1}\circ \phi^{-1}$ (where, due to our abuse of notation, $\phi$ is defined on $\tau_f$). 
Note first that $\phi\circ p_f$ is the exploration process (also denoted $\phi$). 
Then, by Lemma \ref{PiSystemLemma}, it suffices to prove that $\mu=\leb\circ  \phi^{-1}$ on $\mc{L}$. 
However, $\set{t:\imf{\phi}{t}\leq \sigma}=[0,\imf{\mu}{L_\sigma}]$ since $\imf{\phi}{t}=\imf{\phi}{\imf{\mu}{L_{\imf{\phi}{t}}}}$. 
Hence $\mu=\leb \circ \phi^{-1}$ on $\mc{L}$. 
 \end{proof}
 We can finally turn to the proof of Theorem \ref{CRTCodingTheorem}.
 \begin{proof}[Proof of Theorem \ref{CRTCodingTheorem}]
 Existence of a coding function for the tree follows from Proposition \ref{HeightProcessCodesTreeProposition}. 
 
 However, if two trees $\tr{c}_1$ and $\tr{c}_2$ have the same coding function $f$, then Proposition \ref{HeightProcessCodesTreeProposition} also implies that $\tr{c}_1$ and $\tr{c}_2$ are isomorphic to $\tr{c^f}$, so that they are isomorphic to each other. 
\end{proof}


\subsection{Elementary operations on compact TOM trees}
\label{ElementaryOperationsOnCCRTsSection}
 Consider a compact TOM tree $
 \tr{c}=\paren{\paren{\tau,d,\rho},\leq,\mu} $. 
 The objective of this section is to consider two operations on $\tr{c}$ which allow one to construct another TOM tree and we explore the relationships between the contour processes. 
 Consider any element $\sigma$ of $\tau$.
\begin{definition}
 We define the \defin{subtree to the right of $\sigma$}, denoted $\tr{c}_{\geq \sigma}$, equal to $R_\sigma$ with the same root as $\tr{c}$, and with the distance, order and measure obtained by restricting those of $\tr{c}$ to $R_\sigma$. 

We define the \defin{subtree rooted at $\sigma$}, denoted $\tr{c}_{\succeq \sigma}$, equal to $\tau_\sigma=\set{\sigma'\in\tau: \sigma\preceq \sigma'}$ with root $\sigma$, with distance, total order and measure restricted to $\tau_\sigma$.

\end{definition}
\begin{proposition}
For any compact TOM tree $\tr{c}$ and any $\sigma\in\tau$,  $\tr{c}_{\geq \sigma}$ and $\tr{c}_{\succeq \sigma}$ are compact TOM trees. 
\begin{description}
\item[The contour process of $\tr{c}_{\geq \sigma}$]If $\fun{f}{[0,m]}{[0,\infty)}$ is the contour process of $\tr{c}$ and $s^*=\imf{\mu}{L_\sigma}$ then the contour process $f_{\geq \sigma}$ of $\tr{c}_{\geq \sigma}$ is constructed as follows:\begin{linenomath}\begin{esn}
\imf{f_{\geq \sigma}}{t}=\imf{f}{s^*+t}\quad\text{defined on }[0,m-s^*]. 
\end{esn}\end{linenomath}Furthermore, if $\imf{\phi}{s}=\sigma$ then\begin{linenomath}\begin{esn}
s^*=\inf\set{t\geq s: \imf{f}{t}<\imf{f}{s}}=\sup[s]_f. 
\end{esn}\end{linenomath}

\item[The contour process of $\tr{c}_{\succeq \sigma}$]The contour process $f_{\succeq \sigma}$ of $\tr{c}_{\succeq\sigma}$ is defined on $[0,s^*-s_*]$ where\begin{linenomath}\begin{esn}
s^*=\imf{\mu}{L_\sigma}
\text{ and }s_*=\inf\set{\imf{\mu}{L_{\tilde \sigma}}: \tilde\sigma\in\tau_\sigma}. 
\end{esn}\end{linenomath}%
We have the explicit formula\begin{linenomath}\begin{esn}
\imf{f_{\succeq\sigma}}{t}= \imf{f}{t+s_*}-\imf{f}{s^*}\quad\text{ for }t\in [0,s^*-s_*].
\end{esn}\end{linenomath}Furthermore, if $\imf{\phi}{s}=\sigma$ then\begin{linenomath}\begin{esn}
s^*=\inf\set{t\geq s: \imf{f}{t}<\imf{f}{s}}=\sup[s]_f\quad\text{and}\quad s_*=\sup\set{t\leq s: \imf{f}{t}<\imf{f}{s}}=\inf[s]_f. 
\end{esn}\end{linenomath}
%
\end{description}
\end{proposition}
\begin{proof}
Note that $\rho\in R_\sigma$. 
We have seen in Lemma \ref{StrictLeftIsOpenLemma} that $R_\sigma$ is closed. 
If $\sigma\leq \sigma_1\leq \sigma_2$ then $[\sigma_1,\sigma_2]\subset R_\sigma$, so that $R_\sigma$ is a closed and pathwise connected subset of a compact real tree. 
Hence $R_\sigma$, with the induced distance is a compact real tree. 
The total order, when restricted to $R_\sigma$, satisfies properties \defin{Or1} and \defin{Or2}. 
 Restriction to $R_\sigma$ does not alter \defin{Mes1} or \defin{Mes2}. 
 This implies that $\tr{c}_{\geq \sigma}$ is a compact TOM tree. 
 To find its contour process $f_{\geq \sigma}$, note that if $\sigma'\in R_\sigma$ then\begin{linenomath}\begin{esn}
\imf{\mu}{L_{\sigma'}\cap R_\sigma}=\imf{\mu}{L_{\sigma'}}-\imf{\mu}{L_\sigma}=\imf{\mu}{L_{\sigma'}}-s^*.
\end{esn}\end{linenomath}Hence, the exploration process $\phi_{\geq\sigma}$ of $\tr{c}_{\geq \sigma}$ is given by\begin{linenomath}\begin{esn}
\imf{\phi_{\geq \sigma}}{s}=\imf{\phi}{s^*+s}
\end{esn}\end{linenomath}and we deduce the equality $\imf{f_{\geq \sigma}}{s}=\imf{f}{s^*+s}$. 

It remains to prove that if $\imf{\phi}{s}=\sigma$ (or in other words, $s$ is a representative of $\sigma$ when considering $\tr{c}$ as the tree coded by $f$) then\begin{linenomath}\begin{esn}
\imf{\mu}{L_\sigma}=\inf\set{t\geq s: \imf{f}{t}<\imf{f}{s}}. 
\end{esn}\end{linenomath}By Lemma \ref{ExplorationProcessLemma}, between $s$ and $\imf{\mu}{L_{\sigma}}$, $\phi$ explores $\tau_{\sigma}$, so that $f\geq \imf{f}{s}$  on $[s,\imf{\mu}{L_\sigma}]$ which implies the inequality $\imf{\mu}{L_\sigma}\leq s^*$. 
On the other hand, if $s'\in D$ belongs to a right neighborhood of $\imf{\mu}{L_\sigma}$ then $\sigma<\imf{\phi}{s'}$ which implies that $\sigma\wedge \imf{\phi}{s'}\in (\sigma,\rho]$ which in turns tells us that $\imf{m_f}{\imf{\mu}{L_\sigma},s'}<\imf{f}{s}=\imf{f}{s^*}$. 
Hence $s^*\leq s'$ and since $s'$ is an arbitrary element of $D$ and any right neighborhood of $\imf{\mu}{L_\sigma}$, we obtain $s^*\leq\imf{\mu}{L_\sigma}$ and so in fact $s^*=\imf{\mu}{L_\sigma}$.

The last argument is also valid in the setting of $\tr{c}_{\succeq \sigma}$ so that 
$s^*
=\inf\set{t\geq s: \imf{f}{t}<\imf{f}{s}}
$. 
Actually, we have the equality\begin{linenomath}\begin{esn}
\imf{\psi}{\tau_\sigma}=[s_*,s^*]\cap D: 
\end{esn}\end{linenomath}since if $\sigma\preceq\sigma'$ and $\imf{\phi}{s'}=\sigma'$ then $\imf{m_f}{s',s^*}=\imf{f}{s^*}$, so that $\imf{\psi}{\tau_\sigma}\subset [s_*,s^*]\subset D$ and on the other hand, if $s'\in [s_*,s^*]\cap D$ with $s'=\imf{\mu}{L_{\sigma'}}$ then $\imf{m_f}{s',s^*}\geq \imf{f}{s^*}$ by definition of $s_*$ and so $\sigma'\in\tau_\sigma$. 
We now take $s_n$ decreasing to $s_*$ with $s_n=\imf{\psi}{\sigma_n}$ through $D$ and note that\begin{linenomath}\begin{esn}
\tau_\sigma=\bigcup_n L_{\sigma}\setminus L_{\sigma_n}.
\end{esn}\end{linenomath}This implies\begin{linenomath}\begin{esn}
\imf{\mu}{L_{\sigma'}\cap\tau_\sigma}=\lim_{n\to\infty}\imf{\mu}{L_{\sigma'}}-\imf{\mu}{L_{\sigma_n}}=\imf{\psi}{\sigma'}-s_*.
\end{esn}\end{linenomath}for any $\sigma'\in\tau_\sigma$. 

We therefore obtain the following relationship between the exloration process $\phi_{\succeq\sigma}$ of $\tau_\sigma$ and $\phi$:
\begin{linenomath}\begin{esn}
\imf{\phi_{\succeq\sigma}}{s}=\imf{\phi}{s+s_*}.
\end{esn}\end{linenomath}Finally, we note that for any $\sigma'\in\tau_\sigma$\begin{linenomath}\begin{esn}
\imf{d}{\sigma',\sigma}=\imf{d}{\sigma',\rho}-\imf{d}{\sigma',\sigma}
\end{esn}\end{linenomath}to conclude that\begin{linenomath}\begin{esn}
\imf{f_{\succeq\sigma}}{s}=\imf{f}{s}-\imf{f}{s_*}
\end{esn}\end{linenomath}for $s\leq s^*-s_*$.
\end{proof}

\subsection{Topological remarks about the space of compact TOM trees}
Given that compact TOM trees can be identified with a subset of \cadlag\ functions,
 thanks to Theorem \ref{CRTCodingTheorem}, 
 we see that TOM trees constitute a set (without the need of passing to isometry classes). 
Also, we can define the distance between compact TOM trees $\tr{c}_1$ and $\tr{c}_2$ 
in terms of the distance of their contours $f_1$ and $f_2$ as follows. 
Suppose that the supports of $f_i$ is $[0,m_i]$ and that $m_1<m_2$, say. 
First extend $f_1$ to $[0,m_2]$ by declaring it constant on $[m_1,m_2]$. 
We then define\begin{linenomath}\begin{esn}
\imf{d}{\tr{c_1},\tr{c}_2}=\imf{d_{m_2}}{f_1,f_2}+\abs{m_2-m_1}, 
\end{esn}\end{linenomath}where $d_{m_2}$ is the Skorohod $J_1$ distance on $[0,m_2]$ defined as
\begin{linenomath}\begin{esn}
\imf{d_{m_2}}{f_1,f_2}=\sup_{\lambda}\sup_{s\leq m_2}\abs{\imf{f_1}{s}-\imf{f_2\circ \lambda}{s}}
\end{esn}\end{linenomath}
and $\lambda$ runs over all strictly increasing continuous functions of $[0,m_2]$ into itself. 
By uniqueness of the contour in Theorem \ref{CRTCodingTheorem}, 
we see that $d$ is a metric on compact TOM trees. 
If $\tr{c_n}=\paren{\paren{\tau_n,d_n,\rho_n},\leq_n,\mu_n}$ is a sequence of compact TOM trees  
converging to $\tr{c}$ under the metric $d$, 
then, a slight generalization of  Proposition 2.10 of \cite{MR3201925} 
(to cover \cadlag\ functions as done in \cite{saoPauloAmaury}) 
shows us that the rooted measured metric spaces $\paren{\paren{\tau_n,d_n,\rho_n},\mu_n}$ 
converge in the Gromov-Hausdorff-Prokhorov topology associated to the metric $d_{GHP}$ defined as follows:
\begin{linenomath}\begin{align*}
&\imf{d_{GHP}}{\paren{\paren{\tau_1,d_1,\rho_1},\mu_1},\paren{\paren{\tau_n,d_n,\rho_n},\mu_n}}
\\&=\inf_{\phi_1,\phi_2, \tau}
\bra{\imf{d^\tau_H}{\imf{\phi_1}{\tau_1},\imf{\phi_2}{\tau_2}}
+\imf{d^\tau}{\imf{\phi_1}{\rho_1},\imf{\phi_2}{\rho_2}}
+\imf{d^\tau_P}{\mu_1\circ \phi_1^{-1},\mu_2\circ \phi_2^{-1}} }
\end{align*}\end{linenomath}where the infimum runs over all isometric embeddings $\phi_i$ from $\tau_i$ into a common metric space $\tau$, $d^\tau_H$ denotes the Hausdorff distance between subsets of $\tau$, and $d^\tau_P$ stands for the Prokhorov distance between finite measures on $\tau$.

\providecommand{\bysame}{\leavevmode\hbox to3em{\hrulefill}\thinspace}
\providecommand{\MR}{\relax\ifhmode\unskip\space\fi MR }
\providecommand{\MRhref}[2]{%
  \href{http://www.ams.org/mathscinet-getitem?mr=#1}{#2}
}
\providecommand{\href}[2]{#2}

\bibliography{GenBib}
\bibliographystyle{amsalpha}
\end{document}